\documentclass{amsart}
\usepackage{color}
\usepackage{amsmath,amscd}
\usepackage{amssymb}

\ifx\pdfpageheight\undefined
   \usepackage[dvips,colorlinks=true,linkcolor=blue,citecolor=red,%
      urlcolor=green]{hyperref}
   \usepackage[dvips]{graphicx}
   \makeatletter
   \edef\Gin@extensions{\Gin@extensions,.mps}
   \makeatother
\else
 \usepackage[pdftex]{graphicx}
   \usepackage[bookmarksopen=false,pdftex=true,breaklinks=true,%
      backref=page,pagebackref=true,plainpages=false,%
      hyperindex=true,pdfstartview=FitH,colorlinks=true,%
      pdfpagelabels=true,colorlinks=true,linkcolor=blue,%
      citecolor=red,urlcolor=green,hypertexnames=false%
      ]%
   {hyperref}
\fi
\usepackage{bm}

\usepackage{enumerate}
\newtheorem{theorem}{Theorem}[section]
\newtheorem{lemma}[theorem]{Lemma}
\newtheorem{corollary}[theorem]{Corollary}
\newtheorem{conjecture}[theorem]{Conjecture}
\newtheorem{proposition}[theorem]{Proposition}
\theoremstyle{definition}
\newtheorem{definition}[theorem]{Definition}
\newtheorem{example}[theorem]{Example}

\theoremstyle{remark}
\newtheorem{remark}[theorem]{Remark}
\newtheorem{notation}[theorem]{Notation}
\newtheorem{convention}[theorem]{Convention}
\numberwithin{equation}{section}

\newcommand{\Real}{{\mathbb R}}

\newcommand{\eps}{\varepsilon}
\newcommand{\f}{\mathbf{f}}
\newcommand{\x}{\mathbf{x}}
\newcommand{\y}{\mathbf{y}}
\newcommand{\z}{\mathbf{z}}

\newcommand {\hide}[1]{}

\begin{document}
\title[Triangulations of 2D families]
{Triangulations of monotone families I:
Two-dimensional families
}
\thanks{2010 Mathematics Subject Classification 14P10, 14P15, 14P25}
\author{Saugata Basu}
\address{Department of Mathematics,
Purdue University, West Lafayette, IN 47907, USA}
\email{sbasu@math.purdue.edu}
\thanks{S. Basu was partially supported by NSF grants CCF-0915954 and CCF-1319080.}
\author{Andrei Gabrielov}
\address{Department of Mathematics,
Purdue University, West Lafayette, IN 47907, USA}
\email{agabriel@math.purdue.edu}
\thanks{A. Gabrielov and S. Basu were partially supported by  NSF grant DMS-1161629.}
\author{Nicolai Vorobjov}
\address{
Department of Computer Science, University of Bath, Bath
BA2 7AY, England, UK}
\email{nnv@cs.bath.ac.uk}

\begin{abstract}
Let $K \subset \Real^n$ be a compact definable set in an o-minimal structure over $\Real$, e.g.,
a semi-algebraic or a subanalytic set.
A definable family $\{ S_\delta|\> 0< \delta \in \Real \}$ of compact subsets of $K$, is called a {\em monotone family} if
$S_\delta \subset S_\eta$ for all sufficiently small $\delta > \eta >0$.
The main result of the paper is that when $\dim K \le 2$ there exists a definable triangulation of $K$ such that for each
(open) simplex $\Lambda$ of the triangulation
and each small enough $\delta>0$, the intersection $S_\delta \cap \Lambda$ is {\em equivalent}
to one of the five {\em standard} families in the standard simplex (the equivalence relation and a standard family will be
formally defined).
The set of standard families is in a natural bijective correspondence with the set of all
five lex-monotone Boolean functions in two variables.
As a consequence, we prove the two-dimensional case of the topological conjecture in
\cite{GV07} on approximation of definable sets by compact families.
We introduce most technical tools and prove statements for compact sets $K$ of {\em arbitrary} dimensions,
with the view towards extending the main result and proving the topological conjecture in the general case.
\end{abstract}
\maketitle

\section{Introduction}\label{sec:intro}

Let $K \subset \Real^n$ be a compact definable set in an o-minimal structure over $\Real$, for example, it may be
a semi-algebraic or a subanalytic set.
Consider a one-parametric definable family $\{ S_\delta \}_{\delta >0}$ of compact subsets of $K$,
defined for all sufficiently small positive $\delta \in \Real$.

\begin{definition}\label{def:monot_family}
The family $\{ S_\delta \}_{\delta >0}$ is called {\em monotone family} if
the sets $S_\delta$ are monotone increasing as $\delta \searrow 0$, i.e., $S_\delta \subset S_\eta$ for
all sufficiently small $\delta > \eta >0$.
\end{definition}

It is well known that there exists a definable triangulation of $K$ (see \cite{Michel2, Dries}).
In this paper we suggest a more general notion of a definable triangulation of $K$ {\em compatible
with the given monotone family $\{ S_\delta \}_{\delta >0}$}.
The intersection of each set $S_\delta$ with each open simplex of such a triangulation
is a topologically regular cell and is topologically equivalent, in a precise sense, to one of the families
in the finite list of {\em model families}.
Model families are in a natural bijective correspondence with all {\em lex-monotone} Boolean functions in $\dim K$
Boolean variables (see Figure~\ref{fig:1D-2D} for the lists of model families and corresponding lex-monotone
functions in dimensions $1$ and $2$).
We conjecture that such a triangulation always exist, and we prove the conjecture in the case
when $\dim K \leq 2$ (Theorem~\ref{th:main}).

In the course of achieving this goal, we study a problem that is important on its own, of the existence
of a definable cylindrical decomposition of $\Real^n$ compatible with $K$ such that each cylindrical cell of the
decomposition is topologically regular.
Cylindrical decomposition is a fundamental construction in o-minimal geometry \cite{Michel2, Dries},
as well as in semi-algebraic geometry \cite{BCR}.
The elements of a decomposition are called cylindrical cells and are definably homeomorphic to open balls of the
corresponding dimensions.
By definition, a cylindrical decomposition depends on a chosen linear order of coordinates in $\Real^n$.
It is implicitly proved in \cite{Michel2, Dries} that for a given finite collection of definable sets in
$\Real^n$ there is a linear change of coordinates in $\Real^n$ and a cylindrical
decomposition compatible with these sets, such that each cylindrical cell is a topologically regular cell.
Without a suitable change of coordinates, the cylindrical cells defined in various  proofs of existence of cylindrical
decomposition (e.g., in \cite{Michel2, Dries}) can fail to be topologically regular
(see Example 4.3 in \cite{BGV_JEMS11}).

It remains an open problem, even in the category of semi-algebraic sets, whether there always exists a cylindrical
decomposition, with respect to a given order of coordinates, compatible with a given definable bounded set $K$,
such that the cells in the decomposition, contained in $K$, are topologically regular.
We conjecture that such regular cylindrical decompositions always exist, and prove this conjecture in the case
when $\dim K \le 2$ (in this case a weaker result was obtained in \cite{La10}).

Topological regularity is a difficult property to verify in general.
An important tool that we use to prove it for cylindrical cells is the concept of a {\em monotone cell}
introduced in \cite{BGV2} (see Definition~\ref{def:mon_cell} below).
It is proved in \cite{BGV2} that every non-empty monotone cell is a topologically regular cell.
In fact, everywhere in this paper when we prove that a certain cylindrical cell is topologically regular,
we actually prove the stronger property that it is a monotone cell.

\subsection*{History and Motivation}

An important recurring problem in semi-algebraic geometry is to find tight uniform bounds on the topological complexity
of various classes of semi-algebraic sets.
Naturally, in o-minimal geometry, definable sets that are locally closed are easier to handle than arbitrary ones.
A typical example of this phenomenon can be seen in the well-studied problem of obtaining tight upper bounds on
Betti numbers of semi-algebraic or sub-Pfaffian sets in terms of the complexity of formulae defining them.
Certain standard techniques from algebraic topology (for example, inequalities
stemming from the Mayer-Vietoris exact sequence) are directly applicable only in the case of locally
closed definable sets.
Definable sets which are not locally closed are comparatively more difficult to analyze.
In order to overcome this difficulty, Gabrielov and Vorobjov proved in \cite{GV07} the following result.

Suppose that for a bounded definable set $S \subset K \subset \Real^n$ in an o-minimal structure over $\Real$ there is a
definable monotone family $\{S_\delta\}_{\delta >0}$ of compact subsets of $S$  such that
$S= \bigcup_\delta S_\delta$.
Suppose also that for each sufficiently small $\delta >0$ there is a definable family $\{S_{\delta, \eps}\}_{\eps>0}$
of compact subsets of $K$ such that for all $\eps, \eps' \in (0,1)$, if $\eps'>\eps$ then
$S_{\delta, \eps} \subset S_{\delta, \eps'}$, and $S_\delta=\bigcap_\eps S_{\delta, \eps}$.
Finally, assume that for all $\delta'>0$ sufficiently smaller than $\delta$, and all $\eps'>0$ there exists an open
in $K$ set $U \subset K$ such that $S_\delta \subset U \subset S_{\delta', \eps'}$.
The main theorem in \cite{GV07} states that under a certain technical condition on the family $\{S_\delta\}_{\delta>0}$
(called ``separability'' which will be made precise later), for all
$\eps_0 \ll \delta_0 \ll \eps_1 \ll \delta_1 \ll \cdots \ll  \eps_n \ll \delta_n$
(where ``$\ll$'' stands for ``sufficiently smaller than'')
the compact definable set $S_{\delta_0,\eps_0} \cup \cdots \cup S_{\delta_n,\eps_n}$ is homotopy equivalent to  $S$.

The separability condition is automatically satisfied in many cases of interest, such as when $S$ is described by
equalities and inequalities involving  continuous definable functions, and the family $S_\delta$ is defined by replacing
each inequality of the kind $P > 0$ or $P<0$ in the definition of $S$, by $P \geq \delta$ or $P \leq -\delta$
respectively.
However, the property of separability is not  preserved under taking images of definable maps (in particular,
under blow-down maps) which restricts the applicability of this construction.

The following conjecture was made in \cite{GV07}.

\begin{conjecture}\label{conj:equivalence}
The property that the approximating set $S_{\delta_0,\eps_0} \cup \cdots \cup S_{\delta_n,\eps_n}$ is homotopy
equivalent to $S$ remains true even without the separability hypothesis.
\end{conjecture}

Conjecture~\ref{conj:equivalence} would be resolved if one could replace
$S_{\delta_0,\eps_0} \cup \cdots \cup S_{\delta_n,\eps_n}$ by a homotopy equivalent union
$V_{\delta_0,\eps_0} \cup \cdots \cup V_{\delta_n,\eps_n}$ for another, separable, family
$\{ V_{\delta, \eps}\}_{\delta, \eps >0}$, satisfying the same properties as the family ${S_{\delta,\eps}}$ with
respect to $S$.

This motivates the problem of trying to find a finite list of model families inside the standard simplex $\Delta$
such that for each simplex $\Lambda$ of the triangulation of $K$, the family $\{S_\delta \cap \Lambda\}_{\delta>0}$
is topologically equivalent to one of the (separable or non-separable) model families.
Such families $\{S_\delta \cap \Lambda\}_{\delta>0}$ are called {\em standard}.
The main result of this paper is a proof of the existence of a {\em triangulation yielding standard families}
in the two-dimensional case.
As a consequence we obtain a proof of Conjecture~\ref{conj:equivalence} in the case when $\dim K \leq 2$.

This triangulation presents an independent interest.
We will show in Section~\ref{sec:upper-semi-continuous} that there is a bijective correspondence between monotone
families $\{ S_\delta \}_{\delta >0}$ and non-negative upper semi-continuous definable functions
$f:\> K \to \Real$, with $S_\delta=\{\x \in K|\> f(\x) \ge \delta\}$.
Then, for a given $f$, a triangulation into simplices $\Lambda$ yielding standard families
$\{\Lambda \cap S_\delta\}_{\delta>0}$ can be interpreted as a {\em topological resolution of singularities} of
the continuous map $\textrm{graph}(f) \rightarrow K$ induced by $f$,
in the sense that we obtain a partition of the domain into a finite number of simplices  on each of
which the function $f$ behaves in a canonical way up to a certain topological equivalence relation.
A somewhat loose analogy in the analytic setting is provided by
the ``Local Flattening Theorem'' \cite[Theorem 4.4.]{Hironaka_book}.

When $f:\> K \to \Real$ is the distance function to a singular point $\x \in K$, the set $S_\delta$ for small
$\delta >0$ becomes the complement to a neighbourhood of $\x$ in $K$, and the boundary
$\partial S_\delta$ becomes the link of $\x$.
Then the triangulation of $K$, compatible with $\{ S_\delta \}$, can provide a new technique for the
 study of bi-Lipschitz
classification of germs of two-dimensional definable sets \cite{Birbrair}.

\subsubsection*{Relation to triangulations of functions and maps}

It is well known \cite{Michel2, Dries} that continuous definable functions $f:\> K \to \Real$, where $K$ is a compact
definable subset of $\Real^n$, can be triangulated.
A simple example (that of the blow-down map corresponding to the plane $\Real^2$
blown up at a point) shows that definable maps which are not functions (i.e., maps
of the form $f:\> K \to \Real^m,\> m \ge 2$) need not be triangulable, and this leads
to various difficulties in studying topological properties of definable maps.
For example, the question whether a definable map admitting a continuous section,
also admits a definable one would have an immediate positive answer if the map was
definably triangulable.
However, at present this remains a difficult open problem in o-minimal geometry.

The version of the topological resolution of singularities described above can be viewed as
an alternative to the traditional notion of triangulations compatible with a map.
Towards this end, we have identified a special class of definable sets and maps, which
we call semi-monotone sets and monotone maps respectively (see below for
definitions), such that general definable maps could be obtained from these simple
ones via appropriate gluing.

\subsubsection*{Relation to preparation theorems}
An important line of research in o-minimal geometry has been concentrated around preparation theorems.
Given a definable function $f: \Real^{n+1} \rightarrow \Real$, the goal of a preparation theorem
(along the lines of classical preparation theorems in algebra and analysis, due to Weierstrass, Malgrange, etc.)
is to separate the dependence on the last variable, as a power function
with real exponent, from the dependence on the remaining variables.
For example, van den Dries and Speissegger \cite{VDD-Speissegger}, following earlier work by Macintyre, Marker and
Van den Dries \cite{MMV}, Lion and Rolin \cite{Lion-Rolin}, proved that in a  polynomially bounded o-minimal structure
there exists a definable decomposition of $\Real^n$ into definable cells such that over each
cell $C$ the function $f$ can be written as
\[
f(\x,y) = |(y - \theta_C(\x))|^{\lambda_C} g_C(\x) u_C(\x,y).
\]
where $\lambda_C \in \Real$, while $\theta_C, g_C, u_C$ are definable functions with $u_C$ being a unit.
From this viewpoint, the triangulation yielding standard families, could be seen as a topological analogue of a
preparation theorem such as the one mentioned above.
Allowing the unit  $u_C$ in the preparation theorem
gives additional flexibility which is not available in the situation considered in this paper.

\subsection*{Organization of the paper}

Although the main results of the paper are proved in the case when $\dim K \le 2$, most of the definitions and
many technical statements are formulated and proved in the general case.
We consider this paper as the first in the series, and will be using these general definitions and statements in
future work.

The rest of the paper is organized as follows.
In Section~\ref{sec:monotone-cells}, we recall the definition of monotone cells and some of their key properties
needed in this paper.
In Section~\ref{sec:cylindrical}, we recall the definition of definable cylindrical decomposition compatible with
a finite family of definable subsets of $\Real^n$.
The notions of ``top'', ``bottom'' and ``side wall'' of a cylindrical cell that are going to play an
important role later are also defined in this section.
We prove the existence of a cylindrical cell decomposition with monotone cylindrical cells in the case when
$\dim K \leq 2$ (Theorem~\ref{th:cyl_decomp}).

In Section~\ref{sec:upper-semi-continuous}, we establish a connection between monotone definable families of compact
sets, and super-level sets of definable upper semi-continuous functions.
This allows us to include monotone families in the context of cylindrical decompositions.

In Section~\ref{sec:separability}, we recall the notion of ``separability'' introduced in \cite{GV07}, and
discuss certain topological properties of monotone families inside regular cells which will serve as a preparation for
later results on triangulation.

In Section~\ref{sec:standard_and_model} we define the combinatorially standard families and model families.
A combinatorially standard family is a combinatorial equivalence class of monotone families inside the standard simplex
$\Delta^n$.
There is a bijective correspondence between the set of all combinatorially standard families and all
\emph{lex-monotone} Boolean functions on $\{0,1\}^n$ (Definition \ref{def:standard}).
The model families are particular piece-wise linear representatives of the combinatorially standard families
(Definition \ref{def:model}).
After applying a barycentric subdivision to any model family, the monotone family inside each of the
sub-simplices of the barycentric sub-division is guaranteed to be separable (Lemma \ref{le:model_bary}).

In Section~\ref{sec:top-equivalence}, we define the notion of topological equivalence and
prove the existence of certain ``interlacing'' homeomorphisms in the two dimensional case.
This allows us to prove that in two dimensional case
combinatorial equivalence is the same as topological equivalence.

Section~\ref{sec:monotone-curves} is devoted to a technical problem of proving the existence of
monotone curves (and, more generally, families of monotone curves) connecting any two points inside a monotone cell.
Construction of such curves is an essential tool in obtaining a stellar sub-division of a monotone cell into
simplices with an additional requirement that the simplices  are monotone cells.

In Section~\ref{sec:combinatorial-equivalence}, we prove the existence of
a triangulation of two-dimensional compact $K$
such that the restriction of the monotone family to each simplex is standard.

In Section~\ref{sec:triang_sep}, we prove for any given monotone family $\{ S_\delta \}_{\delta>0}$ in
two-dimensional compact $K$ the existence of a
homotopy equivalent monotone family $\{ R_\delta \}_{\delta>0}$ in $K$ and a definable triangulation of $K$ such that
the restriction $\Lambda \cap R_\delta$ to each its simplex $\Lambda$ is separable.

In Section~\ref{sec:approximation}, we prove the motivating conjecture of this paper,
Conjecture~\ref{conj:equivalence}, in the case when the dimension of the set $S$ is at most two
(Theorem~\ref{thm:homotopy}).

\subsection*{Acknowledgements}

The authors thank the anonymous referee for many helpful remarks.

\section{Monotone cells}
\label{sec:monotone-cells}

In \cite{BGV_JEMS11, BGV2} the authors introduced the concepts of a semi-monotone set and a monotone map.
Graphs of monotone maps are generalizations of semi-monotone sets, and will be called {\em monotone cells}
in this paper (see Definition~\ref{def:mon_cell} below).

\begin{definition}
\label{def:semi-monotone}
Let $L_{j, c}:= \{ \x=(x_1, \ldots ,x_n) \in \Real^n|\> x_j = c \}$
for $j=1, \ldots ,n$, and $c \in \Real$.
Each intersection of the kind
$$S:=L_{j_1, c_1} \cap \cdots \cap L_{j_m, c_m} \subset \Real^n,$$
where $m=0, \ldots ,n$, $1 \le j_1 < \cdots < j_m \le n$,
and $c_1, \ldots ,c_m \in \Real$, is called an {\em affine coordinate subspace} in $\Real^n$.

In particular, the space $\Real^n$ itself is an affine coordinate subspace in $\Real^n$.
\end{definition}

We now define \emph{monotone maps}.
The definition below is not the one given in \cite{BGV2}, but equivalent to it as shown in
\cite[Theorem~9]{BGV2}.

We first need a preliminary definition.
For a coordinate subspace $L$ of $\Real^n$ we denote by $\rho_L:\> \Real^n \to L$ the projection map.

\begin{definition}\label{def:quasi-affine}
Let a bounded continuous map $\f=(f_1, \ldots ,f_k)$ defined on an open bounded non-empty set
$X \subset \Real^n$ have the graph
$$Y:=\{(\x, f_1(\x), \ldots, f_k(\x)) \in \Real^{n+k}|\> \x=(x_1, \ldots ,x_n) \in X\}.$$
We say that $\f$ is {\em quasi-affine} if for any coordinate subspace $L$ of $\Real^{n+k}$, the restriction
$\rho_L|_Y$ of the projection is injective if and only if the image $\rho_L(Y)$ is $n$-dimensional.
\end{definition}

\begin{definition}
\label{def:def_monotone_map}
Let a bounded continuous quasi-affine map $\f=(f_1, \ldots ,f_k)$ defined on an open bounded
non-empty set $X \subset \Real^n$ have the graph $Y \subset \Real^{n+k}$.
We say that the map $\f$ is {\em monotone} if for each affine coordinate subspace $S$ in $\Real^{n+k}$ the intersection
$Y \cap S$ is connected.
\end{definition}

\begin{notation}
Let the space $\Real^n$ have coordinate functions $x_1, \ldots, x_n$.
Given a subset $I =\{x_{j_1}, \ldots, x_{j_m} \} \subset \{x_1, \ldots , x_n\}$, let $W$ be the linear subspace
of $\Real^n$ where all coordinates in $I$ are equal to zero.
By a slight abuse of notation we will denote by ${\rm span} \{x_{j_1}, \ldots , x_{j_m}\}$ the
quotient space $\Real^n / W$.
Similarly, for any affine coordinate subspace $S \subset \Real^n$ on which all the functions $x_j \in I$ are
constant, we will identify $S$ with its image under the canonical surjection to $\Real^n /W$.
Again, by a slight abuse of notation, ${\rm span} \{x_1, \ldots x_i \}$, where $i \le n$, will be denoted
by $\Real^i$.
\end{notation}

\begin{definition}[\cite{BGV_JEMS11, BGV2}]\label{def:mon_cell}
A set $Y \subset \Real^n={\rm span}\> \{x_1, \ldots ,x_n \}$ is called a {\em monotone cell}
if it is the graph of a monotone map $\f:\> X \to {\rm span}\> H$, where
$H \subset \{x_1, \ldots ,x_n \}$ and $X \subset {\rm span}(\{x_1, \ldots ,x_n \} \setminus H)$ .
In a particular case, when $H= \emptyset$ (i.e., ${\rm span}\> H$ coincides with the origin)
such a graph is called a {\em semi-monotone set}.
\end{definition}

We refer the reader to \cite{BGV_JEMS11}, Figure 1, for some examples of
monotone cells in $\Real^2$ (actually, semi-monotone sets), as well as some counter-examples.
In particular, it is clear from the examples that the intersection of two monotone cells
in the plane is not necessarily connected and hence not a monotone cell.

Notice that any bounded convex open subset $X$ of $\Real^n$ is a semi-monotone set, while the graph of any
linear function on $X$ is a monotone cell in $\Real^{n+1}$.

The following statements were proved in \cite{BGV2}.

\begin{proposition}[\cite{BGV2}, Theorem~1]\label{Theorem13}
Every monotone cell is a topologically regular cell.
\end{proposition}

\begin{proposition}[\cite{BGV2}, Corollary~7, Theorem~11]\label{prop:def_monotone_map}
Let $X \subset \Real^n$ be a monotone cell.
Then
\begin{enumerate}[(i)]
\item
for every coordinate $x_i$ in $\Real^n$ and every $c \in \Real$,
each of the intersections $X \cap \{ x_i= c \}$, $X \cap \{ x_i< c \}$, $X \cap \{ x_i> c \}$
is either empty or a monotone cell;
\item
Let $Y \subset X$ be a monotone cell such that $\dim Y=\dim X-1$ and $\partial Y \subset \partial X$.
Then $X \setminus Y$ is a disjoint union of two monotone cells.
\end{enumerate}
\end{proposition}

\begin{proposition}[\cite{BGV2}, Theorem~10]\label{prop:proj}
Let $X \subset \Real^n$ be a monotone cell.
Then for any coordinate subspace $L$ in $\Real^n$ the image $\rho_L (X)$
is a monotone cell.
\end{proposition}

Let $\Real^n_{>0}:=\{ \x=(x_1, \ldots ,x_n) \in \Real^n|\> x_i>0\> \text{for all}\> i=1, \ldots ,n \}$,
and $X \subset \Real^n_{>0}$.

\begin{lemma}\label{le:cond_for_semi}
Consider the following two properties, which are obviously equivalent.
\begin{enumerate}[(i)]
\item
For each $\x \in X$, the box
$$B_{\x}:=\{(y_1, \ldots ,y_n) \in \Real^n|\> 0<y_1<x_1, \dots ,0<y_n<x_n\}$$
is a subset of $X$.
\item
For each $\x\in X$ and each $j=1, \ldots ,n$, the interval
$$I_{\x,j}:= \{(y_1, \ldots ,y_n) \in \Real^n|\> 0<y_j<x_j,\> y_i=x_i\> \text{for}\> i\ne j\}$$
is a subset of $X$.
\end{enumerate}
If $X$ is open and bounded, then either of the properties (i) or (ii) implies that $X$ is semi-monotone.
If an open and bounded subset $Y \subset \Real^n_{>0}$ also satisfies the conditions (i) or (ii), then
both $X \cup Y$ and $X \cap Y$ satisfy these conditions, and hence are semi-monotone.
\end{lemma}

\begin{proof}
The proof of semi-monotonicity of $X$ is by induction on $n$, the base for $n=0$ being obvious.
According to Corollary~1 in \cite{BGV2}, it is sufficient to prove that $X$ is connected, and that
for every $k,\> 1 \le k \le n$ and every $c \in \Real$ the intersection $X \cap \{ x_k=c \}$
is semi-monotone.
The set $X$ is connected because for every two points $\x, \z \in X$ the boxes $B_{\x}$ and $B_{\z}$
are connected and $B_{\x} \cap B_{\z} \neq \emptyset$.
Since the property (ii) is true for the intersection $X \cap \{ x_k=c \}$, by the inductive hypothesis
this intersection is semi-monotone, and we proved semi-monotonicity of $X$.

If an open and bounded $Y \subset \Real^n_{>0}$ satisfies the conditions (i) or (ii), then both sets
$X \cup Y$ and $X \cap Y$ obviously also satisfy these conditions, hence are semi-monotone.
\end{proof}

\begin{definition}\label{def:monotone_on_monotone}
Let $Y \subset {\rm span}\> \{x_1, \ldots ,x_n \}$ be a monotone cell and
$\f:\> Y \to {\rm span}\> \{y_1, \ldots ,y_k \}$ a continuous map.
The map $\f$ is called {\em monotone on} $Y$ if its graph
$Z \subset {\rm span}\> \{x_1, \ldots ,x_n,y_1, \ldots ,y_k \}$ is a monotone cell.
In the case $k=1$, the map $\f$ is called a {\em monotone function on} $Y$.
 \end{definition}

\begin{remark}
Let $Y$ be a monotone cell and $L$ a coordinate subspace such that $\rho_L|_{Y}$ is injective.
Then, according to Theorem~7 and Corollary~5 in \cite{BGV2}, $Y$
is the graph of a monotone map defined on $\rho_L(Y)$.
\end{remark}

\section{Cylindrical decomposition}\label{sec:cylindrical}

We now define, closely following \cite{Dries}, a {\em cylindrical cell} and a
{\em cylindrical cell decomposition}.

\begin{definition}\label{def:cyl_cell}
When $n=0$, there is a unique cylindrical cell, $\bf 0$, in $\Real^n$.
Let $n \ge 1$ and $(i_1, \ldots ,i_n) \in \{0,1 \}^n$.
A cylindrical $(i_1, \ldots ,i_n)$-cell is a definable set in $\Real^n$ obtained by induction on $n$
as follows.

A $(0)$-cell is a single point $x \in \Real$, a $(1)$-cell is one of the intervals $(x,y)$ or
$(-\infty, y)$ or $(x, \infty)$ or $(-\infty, \infty)$ in $\Real$.

Suppose that $(i_1, \ldots ,i_{n-1})$-cells, where $n > 1$, are defined.
An $(i_1, \ldots ,i_{n-1},0)$-cell (or a {\em section} cell) is the graph in $\Real^n$ of a
continuous definable function $f:\> C \to \Real$, where $C$ is a $(i_1, \ldots ,i_{n-1})$-cell.
Further, an $(i_1, \ldots ,i_{n-1},1)$-cell (or a {\em sector} cell) is either a set $C \times \Real$,
or a set $\{ (\x,t) \in C \times \Real|\> f(\x) < t <g(\x) \}$, or a set
$\{ (\x,t) \in C \times \Real|\> f(\x) < t \}$, or a set
$\{ (\x,t) \in C \times \Real|\> t <g(\x) \}$, where $C$ is a $(i_1, \ldots ,i_{n-1})$-cell and
$f,g:\> C \to \Real$ are continuous definable functions such that $f(\x)<g(\x)$ for all $\x \in C$.
In the case of a sector cell $C$, the graph of $f$ is called the {\em bottom} of $C$, and
the graph of $g$ is called the {\em top} of $C$.
In the case of a section $(i_1, \ldots ,i_{n-1},0)$-cell $C$, let $k$ be the largest number in
$\{ 1, \ldots ,n-1 \}$ with $i_k=1$.
Then $C$ is the graph of a map $C' \to \Real^{n-k}$, where $C'$ is a sector $(i_1, \ldots ,i_k)$-cell.
The pre-image of the bottom of $C'$ by $\rho_{\Real^k}|_{\overline {C}}$ is called the {\em bottom} of
$C$, and the pre-image of the top of $C'$ by $\rho_{\Real^k}|_{\overline {C}}$ is called the
{\em top} of $C$.
Let $C_T$ be the top and $C_B$ be the bottom of a cell $C$.
The difference $\overline C \setminus (C \cup C_T \cup C_B)$ is called the {\em side wall} of $C$.
\end{definition}

In some literature (e.g., in \cite{Michel2}) section cells are called {\em graphs},
while sector cells -- {\em bands}.

Note that in the case of a sector cell, the top and the bottom are cylindrical section cells.
On the other hand, the top or the bottom of a section cell $C$
is not necessarily a graph of a continuous function
since it may contain blow-ups of the function $\varphi$ of which $C$ is the graph.
Consider the following example.

\begin{example}
\label{eg:blow-up-1}
Let $n=3$, $C' = \{ (x,y)|\> x\in (-1,1), |x| < y <1\}$, and $\varphi(x,y)= |x/y|$.
In this example, the bottom of the cell $C$, defined as the graph of $\varphi|_{C'}$,
is not the graph of a continuous function.
\end{example}

Lemma~\ref{le:top_bottom} below provides a condition under which the top and the bottom of a cylindrical
section cell are cylindrical section cells.

When it does not lead to a confusion, we will sometimes drop the multi-index $(i_1, \ldots ,i_n)$
when referring to a cylindrical cell.

\begin{lemma}\label{le:proj_cyl}
Let $C \subset \Real^n$ be a cylindrical $(i_1, \ldots ,i_{k-1},0,i_{k+1}, \ldots ,i_n)$-cell.
Then
$$C':=\rho_{{\rm span} \{ x_1, \ldots ,x_{k-1},x_{k+1}, \ldots , x_n \}} (C)$$
is a cylindrical $(i_1, \ldots ,i_{k-1},i_{k+1}, \ldots ,i_n)$-cell, and
$C$ is the graph of a continuous definable function on $C'$.
\end{lemma}

\begin{proof}
Proof is by induction on $n$ with the base case $n=1$ being trivial.

By the definition of a cylindrical $(i_1, \ldots ,i_{k-1},0,i_{k+1}, \ldots ,i_n)$-cell,
the image $\rho_{\Real^k } (C)$ is the graph of a continuous function
$\varphi:\> \rho_{\Real^{k-1}} (C) \to {\rm span} \{ x_k \}$.

If $C$ is a section cell, then it is the graph of a continuous function
$f:\> \rho_{\Real^{n -1}} (C) \to {\rm span} \{ x_{n} \}$.
Thus $C'$ is the graph of the continuous function
$$f \circ (x_1, \ldots ,x_{k-1}, \varphi(x_1, \ldots ,x_{k-1}), x_{k+1}, \ldots ,x_{n-1})$$
on $\rho_{{\rm span} \{ x_1, \ldots ,x_{k-1},x_{k+1}, \ldots , x_{n-1} \}}(C)$.
The latter is a cylindrical cell by the inductive hypothesis, since $\rho_{\Real^{n-1}}(C)$ is a
cylindrical $(i_1, \ldots ,i_{k-1},0,i_{k+1}, \ldots ,i_{n-1})$-cell.
Hence $C'$ is a cylindrical cell, being the graph of a continuous function on a cylindrical cell.
By the inductive hypothesis, $\rho_{\Real^{n-1}}(C)$ is the graph of a continuous function $h$ on
$\rho_{{\rm span} \{ x_1, \ldots ,x_{k-1},x_{k+1}, \ldots , x_{n-1} \}}(C)$.
The cell $C$ is the graph of the continuous function $f \circ h \circ \rho_{\Real^{n-1}}|_{C'}$ on $C'$.

If $C$ is a sector cell, then let
$f,\> g: \> \rho_{\Real^{n -1}} (C) \to {\rm span} \{ x_{n} \}$ be its
bottom and its top functions.
Thus, $C'$ is a sector between graphs of functions
$$f \circ (x_1, \ldots ,x_{k-1}, \varphi, x_{k+1}, \ldots ,x_{n-1})\> \text{and}\>
g \circ (x_1, \ldots ,x_{k-1}, \varphi, x_{k+1}, \ldots ,x_{n-1})$$
on the cylindrical cell $\rho_{{\rm span} \{ x_1, \ldots ,x_{k-1},x_{k+1}, \ldots , x_{n-1} \}}(C)$.
Hence $C'$ is a cylindrical cell.
Let $C'_B$ be its bottom and $C'_T$ its top.
The cell $C$ is the graph of the continuous function on $C'$ since the bottom of $C$ is the graph of the continuous
function $f \circ h \circ \rho_{\Real^{n-1}}|_{C'_B}$ on $C'_B$, while the top of $C$ is the graph of the continuous
function $g \circ\ h \circ \rho_{\Real^{n-1}}|_{C'_B}$ on $C'_T$.
Hence each intersection of $C$ with the straight line parallel to $x_n$-axis projects bijectively by
$\rho_{{\rm span} \{ x_1, \ldots ,x_{k-1},x_{k+1}, \ldots , x_n \}}$
onto an intersection of $C'$ with the straight line parallel to $x_n$-axis.
\end{proof}

\begin{lemma}\label{le:side_wall}
Let $C$ be a two-dimensional cylindrical cell in $\Real^n$ such that $C$ is the graph of a quasi-affine map
(see Definition~\ref{def:quasi-affine}).
Then the side wall $W$ of $C$ has exactly two connected components each of which is either a single
point or a closed curve interval.
\end{lemma}

\begin{proof}
Let $C$ be a cylindrical $(i_1, \ldots ,i_n)$-cell and $i_j$ the first 1 in the list $i_1, \ldots ,i_n$.
The image of the projection, $\rho_{{\rm span} \{ x_{j} \}}(C)$ is an interval $(a,b)$.
Consider the disjoint sets
$A:=(\rho_{{\rm span} \{ x_{j} \}}|_{\overline C})^{-1}(a)$ and
$B:=(\rho_{{\rm span} \{ x_{j} \}}|_{\overline C})^{-1}(b)$.
Then $W=A \cup B$, and $A$ (respectively, $B$) is the Hausdorff limit of the intersections
$C \cap \{x_j=c \}$ as $c \searrow a$ (respectively, $c \nearrow b$).
Since $C$ is the graph of a quasi-affine map, for every $c \in (a,b)$ the intersection
$C \cap \{x_j=c \}$ is a curve interval which is also the graph of a quasi-affine map.
Hence each of the Hausdorff limits $A, B$ is either a single point or a closed curve interval.
\end{proof}

\begin{definition}\label{def:cd}
A {\em cylindrical cell decomposition} of $\Real^n$ is a finite partition of $\Real^n$ into cylindrical
cells defined by induction on $n$ as follows.

When $n=0$ the cylindrical cell decomposition of $\Real^n$ consists of a single point.

Let $n>0$.
For a partition $\mathcal D$ of $\Real^n$ into cylindrical cells, let ${\mathcal D}'$
be the set of all cells $C' \subset \Real^{n-1}$ such that $C'= \rho_{\Real^{n-1}}(C)$ for some cell
$C$ of $\mathcal D$.
Then $\mathcal D$ is a cylindrical cell decomposition of $\Real^n$ if ${\mathcal D}'$ is a
cylindrical cell decomposition of $\Real^{n-1}$.
In this case we call ${\mathcal D}'$ the cylindrical cell decomposition of $\Real^{n-1}$ {\em induced} by $\mathcal D$.
\end{definition}

\begin{definition}
\begin{enumerate}[(i)]
\item
A cylindrical cell decomposition $\mathcal D$ of $\Real^n$ is {\em compatible} with a definable set
$X \subset \Real^n$ if for every cell $C$ of $\mathcal D$ either $C \subset X$ or $C \cap X=\emptyset$.
\item
A cylindrical cell decomposition $\mathcal D'$ of $\Real^n$ is a {\em refinement} of a decomposition
$\mathcal D$ of $\Real^n$ if $\mathcal D'$ is compatible with every cell of $\mathcal D$.
\end{enumerate}
\end{definition}

\begin{remark}
It is easy to prove that for a cylindrical cell decomposition $\mathcal D$ of $\Real^n$ compatible with
$X \subset \Real^n$, the cylindrical cell decomposition ${\mathcal D}'$ of $\Real^{n-1}$ induced by $\mathcal D$
is compatible with $\rho_{\Real^{n-1}}(X)$.
\end{remark}

\begin{remark}\label{re:refin}
Let $\mathcal D$ be a cylindrical cell decomposition of $\Real^n$ and $C$
be a cylindrical cell in $\mathcal D$ such that the dimension of $C':=\rho_{\Real^1} (C)$ equals 0, i.e.,
$C'=\{ c \} \subset \Real^1$ for some $c \in \Real$.
It follows immediately from the definitions that $\mathcal D$ is compatible with the hyperplane $\{ x_1=c \}$
in $\Real^n$, and the set of all cells of $\mathcal D$, contained in $\{ x_1=c \}$, forms a cylindrical cell
decomposition ${\mathcal D}'$ of the hyperplane $\{ x_1=c \}$ when the latter is identified with $\Real^{n-1}$.
Moreover, any refinement of ${\mathcal D}'$ leads to a refinement of $\mathcal D$.
\end{remark}

\begin{proposition}[\cite{Dries}, Theorem~2.11]
Let $A_1, \ldots ,A_m \subset \Real^n$ be definable sets.
There is a cylindrical cell decomposition of $\Real^n$ compatible with each of the sets $A_i$.
\end{proposition}

\begin{definition}\label{def:monot_with_respect}
Let $A_1, \ldots ,A_m \subset \Real^n$ be definable bounded sets.
We say that a cylindrical cell decomposition $\mathcal D$ of $\Real^n$ is {\em monotone with respect to}
$A_1, \ldots ,A_m$ if $\mathcal D$ is compatible with $A_1, \ldots ,A_m$, and each cell contained in
$\bigcup_i A_i$ is a monotone cell.
\end{definition}

\begin{lemma}\label{le:section}
Let $\mathcal D$ be a cylindrical cell decomposition of $\Real^n$, and $c \in \Real$.
Then the collection of sets
$$C \cap \{x_1=c \},\> C \cap \{x_1<c \}\> \text{and}\> C \cap \{x_1>c \}$$
for all cylindrical cells $C$ of $\mathcal D$ forms a refinement ${\mathcal D}'$ of $\mathcal D$.
Moreover, for any cylindrical cell $C$ of $\mathcal D$ which is a monotone cell, all cells of ${\mathcal D}'$
contained in $C$ are monotone cells.
\end{lemma}

\begin{proof}
A straightforward induction on $n$, taking into account that intersections of a monotone cell with
a hyperplane or a half-space are monotone cells (Proposition~\ref{prop:def_monotone_map}).
\end{proof}

\begin{definition}
A cylindrical cell decomposition $\mathcal D$ of $\Real^n$ satisfies the {\em frontier condition}
if for each cylindrical cell $C$ its frontier $\overline C \setminus C$ is a union of cells of
$\mathcal D$.
\end{definition}

It is clear that if a cylindrical cell decomposition $\mathcal D$ of $\Real^n$ satisfies the
frontier condition, then the induced decomposition (see Definition~\ref{def:cd}) also satisfies the frontier condition.
It is also clear that the side wall of each cell is a union of some cells in $\mathcal D$ of
smaller dimensions.
We next prove that the tops and the bottoms of cells in a cylindrical decomposition satisfying the
frontier condition are each cells of the same decomposition. Before proving this claim, we first
consider an example.

\begin{example}\label{eg:blow-up-2}
One can easily check that there is a cylindrical cell
decomposition of $\Real^3$ containing the cell $C$ from Example \ref{eg:blow-up-1}
and the cells $\{-1<x<1,\; y=|x|\} \times \Real$ and $\{-1<x<1,\; y=1\} \times \Real$.
This decomposition does not satisfy frontier condition.
The following lemma implies that in fact $C$ cannot be a cell in any
cylindrical decomposition that satisfies the frontier condition.
\end{example}

\begin{lemma}\label{le:top_bottom}
Let $\mathcal D$ be a cylindrical decomposition of $\Real^n$ satisfying the frontier condition.
Then, the top and the bottom of each cell of $\mathcal D$ of $\Real^n$ are cells
of $\mathcal D$.
\end{lemma}

\begin{proof}
Let $C_T$ be the top of a cylindrical cell $C$ of $\mathcal D$.
Suppose $C$ is a sector cell of $\mathcal D$.
By the definition of a cylindrical cell decomposition, $\rho_{\Real^{n-1}}(C)$ is a cylindrical cell.
Since $C_T$ is contained in a union of some cells of $\mathcal D$, and
$\rho_{\Real^{n-1}}(C)= \rho_{\Real^{n-1}}(C_T)$, it is a cylindrical cell of $\mathcal D$.

Suppose now $C$ is a section cell.
Then $C$ is the graph of a map $C' \to \Real^{n-k}$, where $C'$ is a sector cell in the induced cylindrical cell
decomposition ${\mathcal D}'$ of $\Real^k$.
Applying the above argument to $C'$ we conclude that its top $C'_T$ is a cylindrical cell of ${\mathcal D}'$.
By the frontier condition, $C_T$ is a union of some $(k-1)$-dimensional cells of $\mathcal D$.
This is because the pre-image of a cell in a cylindrical cell decomposition is always
a union of cells, and $C_T$ consists of the cylindrical cells in the pre-image $\rho_{\Real^k}^{-1}(C'_T)$
which are contained in $\bar C$, due to the frontier condition.
As $\overline{C} \setminus C$ is $(k-1)$-dimensional, all cells in $C_T$ are $(k-1)$-dimensional and project surjectively
onto $C'_T$, and thus they are disjoint graphs of continuous functions over $C'_T$.

Finally, the closure of a graph of a definable function is a graph of a definable function everywhere except,
possibly over a subset of codimension at least $2$.
Thus $\overline{C}$ cannot contain two disjoint graphs over the $(k-1)$-dimensional cell $C'_T$.

The proof for the bottom of $C$ is similar.
 \end{proof}

\begin{definition}\label{def:side}
Let $\mathcal D$ be a cylindrical decomposition of $\Real^n$ satisfying the frontier condition.
By Lemma~\ref{le:top_bottom}, the top and the bottom of each cell of $\mathcal D$ are cells of $\mathcal D$.
For a cell $C$ of $\mathcal D$ define {\em vertices} of $C$ by induction as follows.
If $\dim C= 0$ then $C$ itself is its only vertex.
Otherwise, the set of vertices of $C$ is the union of the sets of vertices of its top and of its bottom.
\end{definition}

\begin{lemma}\label{le:monot_map}
Let $X$ be an open subset in $\Real^2$, and $\f=(f_1, \ldots ,f_k):\> X \to \Real^k$
a quasi-affine map.
If each component $f_\ell$ is monotone, then the map $\f$ itself is monotone.
\end{lemma}

\begin{proof}
Without loss of generality, assume that none of the functions $f_1, \ldots ,f_k$ is constant.
Let $X \subset {\rm span} \{x_1, x_2\}$ and let
${\bf F} \subset {\rm span} \{x_1, x_2,y_1, \ldots ,y_k \}$ be the graph of $\f$.
Note that for each $i=1, \ldots ,k$ the graph $F_i \subset
{\rm span} \{ x_1, x_2, y_i \}$ of the function $f_i$ coincides with the
image of the projection of $\bf F$ to ${\rm span} \{ x_1,x_2, y_i \}$, and this projection
is a homeomorphism.
By Theorem~9 in \cite{BGV2}, it is sufficient to prove that the intersection of $\bf F$ with
any affine coordinate subspace of codimension 1 or 2 is connected.

First consider the case of codimension 1.
For every $i=1, \ldots ,k$ and every $b \in \Real$ the image of the projection of
${\bf F} \cap \{ y_i=b \}$ to ${\rm span} \{ x_1,x_2, y_i \}$ coincides
with $F_i \cap \{ y_i=b \}$, and this projection is a homeomorphism.
Since $f_i$ is monotone, the intersection $F_i \cap \{ y_i=b \}$ is connected, hence
the intersection ${\bf F} \cap \{ y_i=b \}$ is also connected.
For every $i=1, \ldots ,k$, every $j=1, 2$, and every $c \in \Real$ the image of the
projection of ${\bf F} \cap \{ x_j=c \}$ to ${\rm span} \{ x_1, x_2, y_i \}$
coincides with $F_i \cap \{ x_j=c \}$, thus ${\bf F} \cap \{ x_j=c \}$ is connected since
$F_i \cap \{ x_j=c \}$ is connected.

Now consider the case of codimension 2.
The intersection ${\bf F} \cap \{ x_1=c_1, x_2=c_2 \}$, for any $c_1,c_2 \in \Real$ is
obviously a single point.
The intersection ${\bf F} \cap \{ x_j=c \}$, for any $j=1,2$ and $c \in \Real$ is the graph of
a continuous map on an interval $X \cap \{ x_j=c \}$, taking values in
${\rm span} \{ y_1, \ldots ,y_k \}$ and this map is quasi-affine.
Hence each component of this map is a monotone function.
It follows that the intersection ${\bf F} \cap \{ x_j=c, y_i=b \}$ for every $i=1, \ldots k$
and every $b \in \Real$ is either empty, or a single point, or an interval.
Finally, the intersection ${\bf F} \cap \{ y_i=b_i \}$, for any $i=1, \ldots ,k$ and $b_i \in \Real$
is the graph of a continuous map on the curve $X \cap \{ f_i=b_i \}$ (this curve is
the graph of a monotone function), taking
values in ${\rm span} \{ y_1, \ldots ,y_{i-1},y_{i+1}, \ldots , y_k \}$, and this map is quasi-affine.
Hence each component of this map is a monotone function.
It follows that the intersection ${\bf F} \cap \{ y_i=b_i, y_\ell=b_\ell\}$ for every
$\ell= 1, \ldots ,i-1, i+1, \ldots ,k$ and every $b_\ell \in \Real$
is either empty, or a single point, or an interval.
\end{proof}

\begin{remark}
Let $V_1, \ldots ,V_k$ be bounded definable subsets in $\Real^n$.
According to Section (2.19) of \cite{Dries} (see also Section~4 of \cite{BGV_JEMS11}), there is a cylindrical cell
decomposition of $\Real^n$ compatible with each of $V_1, \ldots ,V_k$, with cylindrical cells being
{\em van den Dries regular}.
One can prove that one- and two-dimensional van den Dries regular cells are topologically regular cells.
Hence, in case $\dim (V_1 \cup \cdots \cup V_k) \le 2$, there exists a cylindrical cell decomposition of $\Real^n$,
compatible with each of $V_i$, such that cylindrical cells contained in $V_1 \cup \cdots \cup V_k$
are topologically regular.
This covers the greater part of the later work \cite{La10}.
Our first goal will be to generalize these results by proving the existence of a cylindrical cell decomposition
of $\Real^n$, monotone with respect to $V_1, \ldots ,V_k$.
\end{remark}

\begin{lemma}\label{le:make_monot}
Let $f:\> X \to \Real$ be a quasi-affine function on an open bounded domain $X \subset \Real^2$.
Then there is a cylindrical cell decomposition of $\Real^2$ compatible with $X$, obtained
by intersecting $X$ with straight lines of the kind $\{ x_1=c \} \subset \Real^2$,
and half-planes of the kind $\{ x_1 \lessgtr c \} \subset \Real^2$, where $c \in \Real$, such that
the restriction $f|_B$ to each cell $B \subset X$ is a monotone function.
\end{lemma}

\begin{proof}
Every non-empty intersection of the kind $X \cap \{x_1=c \}$, where $c \in \Real$, is a finite union of
pair-wise disjoint intervals.
Let ${\mathcal I}(c)$ be family of such intervals.
Let
$$\gamma:= \{ (x_1,x_2) \in X|\> x_2\> \text{is an endpoint of an interval in}\> {\mathcal I}(x_1)\}.$$
Let the real numbers $c_1, \ldots ,c_t$ be such that the intersection $\gamma \cap \{c_i < x_1 <c_{i+1} \}$,
for each $1 \le i <t$, is a disjoint union of monotone 1-dimensional cells with the images under
$\rho_{\Real^1}$ coinciding with $(c_i,c_{i+1})$.
By Theorem~1.7 in \cite{BGV_JEMS11}, the intersection $X \cap \{c_i < x_1 <c_{i+1} \}$ for every $1 \le i <t$
is a disjoint union of one- and two-dimensional semi-monotone sets.
By the definition of $\gamma$, the intersection $X \cap \{x_1=c_i \}$ for every $1 \le i <t$ is
a disjoint union of intervals.
We have constructed a cylindrical decomposition $\mathcal D$ of $\Real^2$ compatible with $X$ and having
semi-monotone cylindrical cells.

Take any two-dimensional cylindrical cell $C$ in $\mathcal D$.
Then $\rho_{\Real^1}(C)=(c_i,c_{i+1})$ for some $1 \le i <t$.
Since $f$ is quasi-affine, its restriction $f_C$ is also quasi-affine, hence (cf. the second part
of Remark~7 in \cite{BGV2}) $f_C$ is either strictly increasing in or strictly decreasing in or independent
of each of the variables $x_1, \ x_2$.
This also implies that the restriction of $f_C$ to any non-empty $C \cap \{ x_2=c \}$, where $c \in \Real$,
is a monotone function.
Let real numbers $b_1, \ldots ,b_r \in (c_i, c_{i+1})$ be such that the restrictions of both functions
$\inf_{x_2}f$ and $\sup_{x_2}f$ to the interval $(b_j,b_{j+1}) \subset \Real^1$ for each $1 \le j <r$
are monotone functions.
Note that the intersection of $\{ b_j< x_1 <b_{j+1} \}$ with any two-dimensional cylindrical cell
in $\mathcal D$ is also a cylindrical (and semi-monotone) cell, in particular the intersection
$B:=C \cap \{ b_j< x_1 <b_{j+1} \}$ is such a cell.
By Theorem~3 in \cite{BGV2}, the restriction $f_B$ is a monotone function.
We have proved that there exists a cylindrical decomposition of $\Real^2$
monotone with respect to $X$ (in particular, the two-dimensional cells of
of the decomposition, contained in $X$, are semi-monotone), and such that the restrictions of $f$ to each cell is
a monotone function.
\end{proof}

\begin{lemma}\label{le:decomp_graphs}
Let $V_1, \ldots ,V_k$ be bounded definable subsets in $\Real^n$ with $\dim V_i \le 2$ for each
$i=1, \ldots ,k$.
Then there is a cylindrical cell decomposition of $\Real^n$, compatible with every $V_i$,
such that every cylindrical cell contained in $V:=\bigcup_{1 \le i \le k} V_i$ is the graph of a
quasi-affine map.
\end{lemma}

\begin{proof}
Let $W$ be the smooth two-dimensional locus of $V$.
Stratify $W$ with respect to critical points of its projections to 2- and 1-dimensional
coordinate subspaces.

More precisely, let $W_{i,j} \subset W$, for $1 \le i < j \le n$, be the set of all locally two-dimensional
critical points of the projection map $\rho_{i,j}:\> W \to {\rm span} \{ x_i,x_j \}$,
and $W'_{i,j}$, for $1 \le i <j \le n$, be the set of all
critical points of the projection map $\rho_{i,j}:\> W \to {\rm span} \{ x_i,x_j \}$,
having local dimension at most 1.

Consider a cylindrical decomposition $\mathcal D$ of $\Real^n$ compatible with each of
$$V_1, \ldots ,V_k, W_{i,j}, W'_{i,j},$$
where $1 \le i < j \le n$.
Let $C$ be a two-dimensional cylindrical cell in this decomposition.
Then $C$ is the graph of a smooth map defined on a cylindrical cell in some 2-dimensional coordinate subspace.
We now prove that this map is quasi-affine.

Note that $C \cap W'_{i,j}= \emptyset$ for all pairs $1 \le i < j \le n$ since $\dim (W'_{i,j})<2$ while
$\mathcal D$ is compatible with $W'_{i,j}$.
Since $\mathcal D$ is compatible with every $W_{i,j}$, if $C \cap W_{i,j} \neq \emptyset$
then $C \subset W_{i,j}$.
Therefore, if $C \cap W_{i,j} \neq \emptyset$ for a pair $i, j,\> 1 \le i < j \le n$, then
$\dim (\rho_{i,j} (C))\leq 1$.

Suppose now that $C \cap W_{i,j}= \emptyset$ for a pair $i, j,\> 1 \le i < j \le n$.
Then $\dim (\rho_{i,j} (C))=2$.
Assume that the projection $\rho_{i,j}|_C$ is not injective, i.e., there are distinct
points ${\bf a}=(a_1, \ldots ,a_n), {\bf b}=(b_1, \ldots ,b_n) \in C$ such that
$\rho_{i,j} ({\bf a})= \rho_{i,j} ({\bf b})$.
There exists $\ell \in \{1, \ldots, n\} \setminus \{i,j \}$ such that $a_\ell \neq b_\ell$.
The set $\rho_{i,j, \ell} (C) \subset {\rm span}\{ x_i,x_j, x_\ell \}$ is two-dimensional, smooth, connected,
and contains points $(a_i, a_j, a_\ell),\> (a_i, a_j, b_\ell)$.
Hence there is a critical point of its projection to the subspace ${\rm span} \{ x_i,x_j \}$.
This contradicts the condition $C \cap (W_{i,j} \cup W'_{i,j})= \emptyset$.
It follows that $C$ is the graph of a quasi-affine map.

Finally, $\mathcal D$ can be refined so that any 1-dimensional cylindrical cell $B$ of the of the refinement
is a monotone (hence quasi-affine) 1-dimensional cell.
Indeed, if $\dim \rho_{\Real^1}(B)=1$, then a refinement exists by Lemma~\ref{le:section}.
Otherwise, $B$ is contained in an affine subspace $\{ x_1=c \}$ for some $c \in \Real$ and a
refinement exists by Remark~\ref{re:refin}.
\end{proof}

\begin{theorem}\label{th:cyl_decomp}
Let $V_1, \ldots ,V_k$ be bounded definable subsets in $\Real^n$ with $\dim V_i \le 2$ for each
$i=1, \ldots ,k$.
Then there is a cylindrical cell decomposition of $\Real^n$ satisfying the frontier condition,
and monotone with respect to $V_1, \ldots ,V_k$.
\end{theorem}

\begin{proof}
First, using Lemma~\ref{le:decomp_graphs}, construct a cylindrical cell decomposition $\mathcal D$ of
$\Real^n$, compatible with every $V_i$, such that each cylindrical cell contained in
$V$ is the graph of a quasi-affine map.

We now construct, inductively on $n$, a refinement of $\mathcal D$, which is a cylindrical cell
decomposition with every cell contained in $V:=\bigcup_{1 \le i \le k} V_i$ being a monotone cell.
The base case $n=1$ is straightforward.
Suppose the construction exists for all dimensions less than $n$.
Each cylindrical cell $X$ in $\mathcal D$ such that $\dim \rho_{\Real^1}(X)=0$ belongs to a
cylindrical cell decomposition in the $(n-1)$-dimensional affine subspace $\{ x_1=c \}$ for some
$c \in \Real$, and in this subspace the refinement can be carried out by the inductive hypothesis.
According to Remark~\ref{re:refin}, this refinement is also a refinement of $\mathcal D$.
Now let $X$ be a cylindrical cell in ${\mathcal D}$, contained in $V$, with $\dim \rho_{\Real^1}(X)=1$.
If $\dim X=1$, then $X$, being quasi-affine, is already a monotone cell.

Suppose $\dim X=2$.
Let $\alpha$ be the smallest number among $\{ 1, \ldots ,n \}$ such that
$X':=\rho_{{\rm span} \{x_1, x_\alpha \}}(X)$ is two-dimensional.
Then $X$ is a graph of a quasi-affine map $\f=(f_1, \ldots ,f_{n-2})$ defined on $X'$.
Since $\f$ is quasi-affine, each $f_j$ is quasi-affine too.
By Lemma~\ref{le:make_monot}, for each $f_j$ there exists a cylindrical decomposition of
${\rm span}\{ x_1, x_\alpha \}$, compatible with $X'$, obtained
by intersecting $X'$ with straight lines of the kind $\{ x_1=c \} \subset {\rm span}\{ x_1, x_\alpha \}$,
and half-planes of the kind $\{ x_1 \lessgtr c \} \subset {\rm span}\{ x_1, x_\alpha \}$,
where $c \in \Real$, such that the
restriction $f_j|_{Y'}$ for each cylindrical semi-monotone cell $Y' \subset X'$ is a monotone function.
According to Lemma~\ref{le:section}, the intersections of all cylindrical cells in ${\mathcal D}$ with
$\{ x_1=c \}$ or $\{ x_1 \lessgtr c \}$ form a cylindrical cell decomposition.
Performing such a refinement for each $f_j$ we obtain a cylindrical cell decomposition of $X'$ into
cylindrical cells $Y''$, such that the restriction $\f|_{Y''}$ is a monotone map
by Lemma~\ref{le:monot_map}.
Therefore all elements of the resulting cylindrical cell decomposition, contained in $X$, are monotone cells.

Decomposing in this way each two-dimensional set $X$ of ${\mathcal D}$ we obtain a refinement
${\mathcal D}'$ of ${\mathcal D}$ which is a cylindrical cell decomposition with monotone
cylindrical cells.

It remains to construct a refinement of the cylindrical cell decomposition ${\mathcal D}'$ satisfying
the frontier condition.
Let $X$ be a two-dimensional cylindrical cell in ${\mathcal D}'$.
Since $X$ is a monotone cell, its boundary $\partial X$ is homeomorphic to a circle.
Let $\mathcal U$ be a partition of $\partial X$ into points and curve intervals so that $\mathcal U$ is
compatible with all 1-dimensional cylindrical cells of ${\mathcal D}'$, and each curve interval
in $\mathcal U$ is a monotone cell.

Let ${\bf c}=(c_1, \ldots, c_n)$ be point (0-dimensional element) in $\mathcal U$ such that $c_1$
is not a 0-dimensional cell in the cylindrical decomposition induced by ${\mathcal D}'$ on $\Real^1$.
By Lemma~\ref{le:section}, intersections of the cylindrical cells of ${\mathcal D}'$ with
$\{x_1=c_1 \}$ or $\{x_1 \lessgtr c_1 \}$ form the refinement of ${\mathcal D}'$ with cylindrical
cells remaining to be monotone cells.
Let $T \in {\mathcal U}$ be one of monotone curve intervals having $\bf c$ as an endpoint.
If $T$ is a subset of a two-dimensional cylindrical cell $Z$ of ${\mathcal D}'$,
then $T$ divides $Z$ into two two-dimensional cylindrical cells, hence by Theorem~11 in \cite{BGV2},
these two cells are monotone cells.
Obviously, adding $T$ to the decomposition, and replacing one two-dimensional cell $Z$ (if it exists) by two cells,
we obtain a refinement of ${\mathcal D}'$.

Let ${\bf c}=(c_1, \ldots, c_n)$ be point in $\mathcal U$ such that $(c_1, \ldots, c_{i-1})$, where $i<n$, is
a 0-dimensional cell in the cylindrical decomposition induced by ${\mathcal D}'$ on $\Real^{i-1}$, while
$(c_1, \ldots, c_i)$ is not a 0-dimensional cell in the cylindrical decomposition induced by ${\mathcal D}'$
on $\Real^i$.
In this case we apply the same construction as in the previous case, replacing $\Real^n$ by
$\{ x_1=c_1, \ldots ,x_{i-1}=c_{i-1} \}$.
By Remark~\ref{re:refin}, the refinement in $\{ x_1=c_1, \ldots ,x_{i-1}=c_{i-1} \}$ is also
a refinement of ${\mathcal D}'$.

Application of this procedure to all two-dimensional cells $X$ of ${\mathcal D}'$, all $\bf c$ and all $T$ completes
the construction of a refinement of ${\mathcal D}'$ which satisfies the frontier condition.
\end{proof}

\begin{corollary}\label{cor:refinement}
Let $U_1, \ldots ,U_m$ be bounded definable subsets in $\Real^n$ with $\dim U_i \le 2$ for each
$i=1, \ldots ,m$ and let ${\mathcal A}$ be a cylindrical decomposition of $\Real^n$.
Then there is a refinement of ${\mathcal A}$, satisfying the frontier condition,
and monotone with respect to $U_1, \ldots ,U_m$.
\end{corollary}

\begin{proof}
Apply Theorem~\ref{th:cyl_decomp} to the family $V_1, \ldots ,V_k$ consisting of sets
$U_1, \ldots ,U_m$ and all cylindrical cells of the decomposition ${\mathcal A}$.
\end{proof}

The following example shows that there may not exist a cylindrical cell decomposition of $\Real^3$
compatible with a two-dimensional definable subset, such that each component of the side wall of each two-dimensional
cell is a one-dimensional cell of this decomposition.
We will call the latter requirement the {\em strong frontier condition}.

\begin{example}
Let $V= \{ x>y>0,\> z>0,\> y=xz \}$ and $V'=\{x>y>0,\> z>0,\> y=2xz\}$ be two cylindrical cells in $\Real^3$.
Then any cylindrical decomposition of $\Real^3$ compatible with $V$ and $V'$ and satisfying
the strong frontier condition, must be compatible with two intervals
$I_1:= \{ x=y=0,\> 0 \le z \le 1/2 \}$ and $I_2:= \{ x=y=0,\> 1/2 \le z \le 1 \}$ on the $z$-axis,
and the point $v=(0,\ 0,\ 1/2)$.
Observe that the interval $I:=\{ x=y=0,\> 0 \le z \le 1 \}$ is the (only) 1-dimensional component of the side wall
of $V$, while $I_1$ is the (only) 1-dimensional component of the side wall of $V'$.

In order to satisfy the strong frontier condition, we have to partition $V$ into cylindrical cells so that there is a
1-dimensional cell $\gamma\subset V$ such that $v=\overline{\gamma}\cap I$.
Then the tangent at the origin of the projection $\beta:=\rho_{\Real^2}(\gamma)$ would have slope $1/2$.
The lifting $\gamma'$ of $\beta$ to $V'$ (i.e., $\gamma' := (\rho_{\Real^2}|_{V'})^{-1}(\beta)$)
would satisfy the condition $v'=\overline{\gamma'}\cap I$, where $v'=(0,0,1/4)$.
(Note that the tangent to $\gamma$ at $v$ or the tangent to $\gamma'$ at $v'$ may coincide with the
$z$-axis.)

The point $v'$ must be a 0-dimensional cell of the required cell decomposition.
Iterating this process, we obtain an infinite sequence of points $(0,\ 0,\ 2^{-k})$, for all
$k>0$, on $I$, all being 0-dimensional cells of a cylindrical cell decomposition.
This is a contradiction.
\end{example}

\begin{lemma}\label{le:monot_cell}
Let $X$ be a cylindrical {\em sector} cell in $\Real^n$ with respect to the ordering $x_1, \ldots , x_n$
of coordinates.
Suppose that the top and the bottom of $X$ are monotone cells.
Then $X$ itself is a monotone cell.
\end{lemma}

\begin{proof}
Let $X':= \rho_{\Real^{n-1}}(X)$.
Then $X= \{ (x_1, \ldots ,x_n) \in \Real^n|\> \x:=(x_1, \ldots ,x_{n-1}) \in X',\>
f(\x) < x_n < g(\x) \}$, where $f,g:\> X' \to \Real$ are monotone functions, having graphs $F$
and $G$ respectively.
Note that $\rho_{\Real^{n-1}} (F)=\rho_{\Real^{n-1}}(G)=X'$.
According to Theorem~10 in \cite{BGV2}, $X'$ is monotone cell.
It easily follows from Theorem~9 in \cite{BGV2} that for any $a \in \Real$ the cylinder
$C:=(X' \times \Real) \cap \{-|a|< x_n <|a| \}$ is a monotone cell.
Choose $a$ so that $-|a| < \inf_{x_n} f$ and $|a| > \sup_{x_n} g$.
Then we have the following inclusions: $F \subset C,\> G \subset C,\> \partial F \subset \partial C$
and $\partial G \subset \partial C$.
By Theorem~11 in \cite{BGV2}, the set $X$ is a monotone cell, being a connected
component of $C \setminus (F \cup G)$.
\end{proof}

The following statement is a generalization of the main result of \cite{La10}.

\begin{corollary}\label{cor:cyl_decomp}
Let $U_1, \ldots ,U_k$ be bounded definable subsets in $\Real^n$, with $\dim U_i \le 3$.
Then there is a cylindrical decomposition of $\Real^n$, compatible with each $U_i$, such that
\begin{enumerate}[(i)]
\item
for $p \le 2$ each $p$-dimensional cell $X \subset \Real^n$ of the decomposition,
contained in $U:= \bigcup_i U_i$, is a monotone cell;
\item
each 3-dimensional {\em sector} cell $X \subset \Real^n$ of the decomposition, contained in $U$,
is a monotone cell;
\item
if $n=3$, then each 3-dimensional cell, contained in $U$, is a semi-monotone set and all cells
of smaller dimensions, contained in $U$, are monotone cells.
\end{enumerate}
\end{corollary}

\begin{proof}
Construct a cylindrical decomposition $\mathcal D$ of $\Real^n$ compatible with each $U_i$, and let
$V_1, \ldots ,V_r$ be all 0-, 1- and 2-dimensional cells contained in $U$.
Using Theorem~\ref{th:cyl_decomp} obtain a cylindrical decomposition
${\mathcal D}'$ of $\Real^n$ monotone with respect to $V_1, \ldots ,V_r$.

Observe that the decomposition ${\mathcal D}'$ is a refinement of $\mathcal D$ and hence
is compatible with each $U_i$.
Therefore (i) is satisfied.
Since for each 3-dimensional {\em sector} cell $X$, contained in $U$, its top and its bottom are
monotone cells, Lemma~\ref{le:monot_cell} implies that $X$ is itself a monotone cell,
and thus (ii) is satisfied.

If $n=3$, then every 3-dimensional cell in $\mathcal D$ is a sector cell, which implies (iii).
\end{proof}

\section{Monotone families as superlevel sets of definable functions}
\label{sec:upper-semi-continuous}
\begin{convention}\label{con:monot_family}
In what follows we will assume that for each monotone family $\{ S_\delta \}$ in a compact definable set
$K \subset \Real^n$ (see Definition~\ref{def:monot_family})
there is $\delta_1 >0$ such that $S_\delta =\emptyset$ for all $\delta > \delta_1$.
\end{convention}

\begin{lemma}\label{le:germ}
Let  $K \subset \Real^n$ be a compact definable set, and
$\{ S_\delta \}_{\delta >0}$  a monotone definable family of compact subsets of $K$. There exists
$\delta_0 > 0$, such that the monotone definable family $\{S'_\delta\}_{\delta > 0}$ defined by
$$
S'_\delta = \begin{cases} S_\delta & \textrm{for}\quad 0<\delta\le\delta_0\\ \emptyset &
\textrm{for}\quad\delta>\delta_0 \end{cases}
$$
has the following property.
For each $\x \in K$ let $M_{\x}:=\{ \delta \in \Real_{>0}|\> \x \in S'_\delta \}$.
Then, either $M_\x = \emptyset$, or $M_\x = (0,b_\x]$ for some $b_\x > 0$.
\end{lemma}

\begin{proof}
By Hardt's theorem for definable families \cite[Theorem 5.22]{Michel2},
there exists $\delta_0 > 0$, and a fiber-preserving homeomorphism
$H:  (0,\delta_0] \times S_{\delta_0} \rightarrow  S_{(0,\delta_0]}$, where for any subset $I \subset \Real$,
$S_{I} := \{(\delta,\x)|\> \delta \in I , \x \in S_\delta \} \subset \Real^{n+1}$.
Let $\x \in K$ be such that $M_\x$ is not empty.
Then, there exists $c, \;0 < c \le \delta_0$, such that  $\x \in S'_c$.
Then,
$$
M_\x=(0,c] \cup \left\{\delta \in [c,\delta_0]\>|\> \x \in S'_\delta \right\}=
(0,c] \cup \left\{\delta \in [c,\delta_0]\>|\> (\delta,\x) \in H^{-1}\left(S'_{[c,\delta_0]}\right)\right\}.
$$

Now $S'_{[c,\delta_0]}$ is compact, and hence $H^{-1}\left(S'_{[c,\delta_0]}\right)$ is also compact, and since
the projection of a compact set is compact,
$\left\{\delta \in [c,\delta_0] \;\mid\; (\delta,\x) \in H^{-1}\left(S'_{[c,\delta_0]}\right)\right\}$
is compact as well.
\end{proof}

\begin{convention}\label{con:identify_families}
In what follows we identify two families $\{ S_\delta \}$ and $\{ V_\delta \}$ if $S_\delta=V_\delta$
for small $\delta >0$.
In particular, the families $\{ S_\delta \}$ and $\{ S'_\delta \}$ from Lemma~\ref{le:germ} will be identified.
\end{convention}

With any monotone definable family $\{S_\delta\}_{\delta>0}$ of compact sets contained in a compact definable set $K$
(see Definition~\ref{def:monot_family}),
we associate a definable non-negative upper semi-continuous function $f:\> K \to \Real$ as follows.

\begin{definition}
\label{def:association}
Associate with the given family $\{ S_\delta \}_{\delta >0}$ the family
$\{ S'_\delta \}_{\delta >0}$ satisfying the conditions of Lemma \ref{le:germ}.
Define for each $\x \in K$ the value $f(\x)$ as $\max \{ \delta|\> \x \in S'_\delta \}$,
if there is $\delta >0$ with $\x \in S'_\delta$, or as $0$ otherwise
(by Lemma \ref{le:germ}, the function $f$ is well-defined).
Note that, by Convention~\ref{con:monot_family}, the function $f$ is bounded.
\end{definition}

\begin{convention}\label{con:identify}
We identify any two non-negative functions $f,g:K\rightarrow \Real$ if they have the same
level sets $\{ \x \in K|\> f(\x)=\delta \}=\{ \x \in K|\> g(\x)=\delta \}$
for small $\delta>0$.
\end{convention}

\begin{lemma}\label{le:S_by_function}
For a compact definable set $K \subset \Real^n$, there is a bijective correspondence between monotone definable families
$\{ S_\delta \}_{\delta >0}$ of compact subsets of $K$ and non-negative upper semi-continuous definable functions
$f:\> K \to \Real$, with $S_\delta=\{\x \in K|\> f(\x) \ge \delta\}$.
\end{lemma}

\begin{proof}
The bijection is defined in Definition \ref{def:association}. The lemma follows
from the fact that a function $\psi:\> X \to \Real$ on a topological
space $X$ is upper semi-continuous
if and only if the set $\{ \x \in X|\> \psi (\x) < b \}$ is open for every $b \in \Real$,
and the identifications made in Conventions~\ref{con:identify_families} and \ref{con:identify}.
\end{proof}

\begin{remark}\label{re:S_by_function}
Observe that due to the correspondence in Lemma~\ref{le:S_by_function}, the union
$S:=\bigcup_{\delta} S_\delta$ coincides with $\{ \x \in K|\> f(\x) >0 \}$,
the complement $K \setminus S$ coincides with the 0-level set $\{ \x \in K|\> f(\x)=0 \}$
of the function $f$.
\end{remark}

\begin{lemma}
For a compact definable set $K \subset \Real^n$,
there is a bijective correspondence between arbitrary non-negative definable functions
$h:\> K \to \Real$ and monotone definable families $\{ S_\delta \}_{\delta>0}$
of subsets of $K$, with $S_\delta= \{ \x \in K|\> h(\x) \ge \delta \}$, satisfying the
following property.
There exists a cylindrical decomposition
${\mathcal D}$ of $\Real^n$, compatible with $K$, such that for small $\delta >0$ and
every cylindrical cell $X$ in ${\mathcal D}$, the intersection $S_\delta \cap X$ is closed in $X$.
\end{lemma}

\begin{proof}
Let $h:\> K \to \Real$ be a non-negative definable function.
Consider a cylindrical decomposition ${\mathcal D}'$ of $\Real^{n+1}$ compatible with
$K$ and the graph of $h$ in $\Real^{n+1}$.
Note that ${\mathcal D}'$ induces a cylindrical decomposition ${\mathcal D}$ of $\Real^n$
compatible with $K$.
Then the family $\{\{ \x \in K|\> h(\x) \ge \delta \}\}_{\delta >0}$ and the decomposition
${\mathcal D}$ satisfy the requirements, since by Definition~\ref{def:cd} (of a cylindrical decomposition),
for every cell $X$ the restriction $h|_X$ is a continuous function.

Conversely, given a family $\{ S_\delta \}_{\delta>0}$ and a cylindrical decomposition
${\mathcal D}$ such that for small $\delta>0$ and
every cylindrical cell $X$ in ${\mathcal D}$, the intersection $S_\delta \cap X$ is closed in $X$,
consider the family of compact sets  $\{ \overline {S_\delta \cap X} \}_{\delta >0}$ in
$\overline X$ (note that $\overline {S_\delta \cap X} \cap X= S_\delta \cap X$, since $S_\delta$ is
closed in $X$).
Applying Lemma~\ref{le:S_by_function} to  $\{ \overline {S_\delta \cap X} \}_{\delta >0}$, we
obtain an upper semi-continuous function $h_X:\> \overline X \to \Real$ such that
$\overline {S_\delta \cap X}= \{\x \in \overline X|\> h_X(\x) \ge \delta\}$, and therefore,
$S_\delta \cap X= \{ \x \in X|\> h_X(\x) \ge \delta\}$.
The function $h:\> K \to \Real$ is now defined by the restrictions of $h_X$ on all
cylindrical cells $X$ of ${\mathcal D}$.
\end{proof}

\begin{definition}\label{def:monot_respect_function}
Let $K \subset \Real^n$ be a compact definable set, $\{ S_\delta \}_{\delta>0}$ a monotone definable family of
compact subsets of $K$, and $f$ the corresponding non-negative upper semi-continuous definable function.
Let $F$ be the graph of the function $f$.
We say that a cylindrical cell decomposition ${\mathcal C}$ of $\Real^{n+1}$ is {\em monotone with respect to
the function} $f$ if ${\mathcal C}$ is monotone with respect to sets $F$ and $\{ x_{n+1}=0 \}$.
\end{definition}

\begin{remark}\label{re:S_by_monot}
By Theorem~\ref{th:cyl_decomp}, for each $K$ and $\{ S_\delta \}_{\delta>0}$ there exists a cylindrical cell
decomposition of $\Real^{n+1}$ satisfying the frontier condition and monotone with respect to $f$.
Let $\mathcal D$ be the cylindrical decomposition induced by ${\mathcal C}$ on $\Real^n$.
Then
\begin{enumerate}[(i)]
\item
all cylindrical cells $Z$ of ${\mathcal C}$ contained in the graph $F$ and all cylindrical cells $Y$ of
${\mathcal D}$ contained in $K$ are monotone cells;
\item
for every $Y$ of $\mathcal D$ contained in $K$, the restriction $f|_Y$ is a monotone function on $Y$
(see Definition~\ref{def:monotone_on_monotone}), either positive or identically zero, and such that
$$S_\delta \cap Y = \{ \x \in Y|\> f|_Y(\x) \ge \delta \}$$
for small $\delta >0$.
\end{enumerate}
Indeed, each cylindrical cell $Z$ of ${\mathcal C}$ is a monotone cell hence, by Proposition~\ref{prop:proj},
each cylindrical cell $Y$ in ${\mathcal D}$ is a monotone cell.
Since the graph of the restriction $f|_Y$ is a cylindrical cell in ${\mathcal C}$, it is a monotone cell.
The function $f|_Y$ is either positive or identically zero due to Remark~\ref{re:S_by_function}.
\end{remark}

\section{Separability and Basic Conditions}
\label{sec:separability}
\begin{definition}
By the {\em standard $m$-simplex} in $\Real^m$ we mean the set
$$\Delta^m:= \{ (t_1, \ldots ,t_m) \in \Real_{>0}^m|\> t_1 + \cdots +t_m < 1 \}.$$
We will assume that the vertices of $\Delta^m$ are labeled by numbers $0, \ldots ,m$ so that
the vertex at the origin has the number $0$, while the vertex with $x_i=1$ has the number $i$.
We will be dropping the upper index $m$ for brevity, in cases when this does not lead to a confusion.
\end{definition}

\begin{definition}
An {\em ordered $m$-simplex in} $\Real^n$ is an open simplex with some total order on its vertices.
An {\em ordered simplicial complex} is a finite simplicial complex, such that all its simplices
are ordered and the orders are compatible on the faces of simplices.
\end{definition}

\begin{remark}
For each ordered $m$-simplex $\Sigma$ there is a canonical affine map from $\Delta^m$ to a standard simplex
$\Sigma$ preserving the order of vertices.
\end{remark}

\begin{definition}
An {\em ordered definable triangulation} of a compact definable set $K \subset \Real^n$ {\em compatible
with subsets} $A_1, \ldots ,A_r \subset K$ is a definable
homeomorphism $\Phi:\> C \to K$, where $C$ is an ordered simplicial complex, such that each
$A_i$ is the union of images by $\Phi$ of some simplices in $C$.
\end{definition}

\begin{proposition}[\cite{Michel2}]
Let $K$ be a compact definable subset in $\Real^n$ and $A_1, \ldots ,A_r$ be definable subsets in $K$.
Then there exists an ordered definable triangulation of $K$ compatible with $A_1, \ldots ,A_r$.
\end{proposition}

\begin{definition}\label{def:simplex}
For $m \le n$, by an {\em ordered definable $m$-simplex} we mean a pair $( \Lambda, \Phi)$ where
$\Lambda \subset \Real^n$ is a bounded definable set, and
$\Phi:\> |C| \to \overline \Lambda$ is an ordered definable triangulation of its closure $\overline \Lambda$,
where $C$ is the complex consisting of a standard $m$-simplex $\Delta$ and all of its faces,
such that $\Phi(\Delta)=\Lambda$.
The images by $\Phi$ of the faces of $\Delta$ are called {\em faces of} $\Lambda$.
Zero-dimensional faces are called {\em vertices of} $\Lambda$.
If $(\Sigma, \Psi)$ is another definable $m$-simplex in $\Real^n$, then
a homeomorphism $h:\> \overline\Lambda \to \overline\Sigma$ is called {\em face-preserving} if
$\Psi^{-1} \circ h \circ \Phi$ is a face-preserving homeomorphism of $\overline \Delta$
(i.e., sends each face to itself).
\end{definition}

\begin{convention}
In what follows, whenever it does not lead to a confusion, we will assume that for a given
bounded definable set $\Lambda$ the map $\Phi$ is fixed, and will refer to an ordered definable simplex
by just $\Lambda$.
\end{convention}

Obviously, in the definable triangulation of any compact definable set the image by $\Phi$ of any
simplex of the simplicial complex is a definable simplex.

Let $\{ S_\delta \}_{\delta>0}$ be a monotone definable family  of subsets of a compact
definable set $K \subset \Real^n$, and $f:\> K \to \Real$
the corresponding upper semi-continuous function, so that $S_\delta=\{\x \in K|\> f(\x) \ge \delta\}$.

\begin{convention}\label{con:family_in_simplex}
In what follows, for a given definable simplex $\Lambda$ we will write, slightly abusing the notation,
$S_\delta$ instead of $S_\delta \cap \Lambda$, and will say that
$S_\delta$ is a {\em family in} $\Lambda$.
A family $S_\delta$ is {\em proper} if neither $S_\delta = \Lambda$, nor $S_\delta = \emptyset$.
\end{convention}

\begin{notation}
For a subset $A \subset \Lambda$ of a definable simplex $\Lambda$ we will use the following
terms and notations.
\begin{itemize}
\item
The interior ${\rm int} (A):= \{ \x \in A|\> \text{ a neighbourhood of}\> \x\> \text{in}\> \Lambda\>
\text{lies in}\> A \}$.
\item
The boundary $\partial A:= \overline A \setminus {\rm int} (A)$.
\item
The boundary in $\Lambda$, $\partial_\Lambda A:=\partial A \cap \Lambda$.
\end{itemize}
\end{notation}

\begin{definition}
The set $\{ f=\delta \} \cap \partial_\Lambda S_\delta$ is called the {\em moving part} of
$\partial_\Lambda S_\delta$, the set $\{ f>\delta \} \cap \partial_\Lambda S_\delta$ is called the
{\em stationary part} of $\partial_\Lambda S_\delta$.
\end{definition}

\begin{remark}\label{re:moving}
Obviously the moving part and the stationary part form a partition of $\partial_\Lambda S_\delta$.
It is easy to see that the stationary part of $\partial_\Lambda S_\delta$ coincides with
$$\{\x \in \Lambda|\> \liminf_{\y \in \Lambda,\, \y \to \x} f(\y) \le \delta <f(\x)\}.$$
At each point of stationary part of the boundary, the function $f|_\Lambda$ is discontinuous.
In particular, if $f|_\Lambda$ is continuous then the stationary part is empty.
\end{remark}

\begin{definition}[\cite{GV07}, Definition~5.7]\label{def:separable2}
A family $S_\delta$ in a definable $n$-simplex $\Lambda$ is called {\em separable} if for
small $\delta >0$ and every face $\Lambda'$ of $\Lambda$, the inclusion
$\Lambda' \subset \overline{\Lambda \setminus S_\delta}$ is equivalent to
$\overline S_\delta \cap \Lambda'= \emptyset$.
\end{definition}

\begin{remark}\label{re:separable2}
Observe that the implication
$$(\overline S_\delta \cap \Lambda'= \emptyset) \Rightarrow
(\Lambda' \subset \overline{\Lambda \setminus S_\delta})$$
is true regardless of separability,
since $\overline S_\delta \cap \Lambda'= \emptyset$ is equivalent to $\Lambda' \subset
\overline \Lambda \setminus \overline S_\delta$, while
$\overline \Lambda \setminus \overline S_\delta \subset \overline{\Lambda \setminus S_\delta}$.
\end{remark}

\begin{notation}
It will be convenient to label the face of an ordered definable simplex $\Lambda$ opposite to vertex
$j$ by $\Lambda_j$, and
the face of $\Lambda$ opposite to vertices $j_1, \ldots ,j_k$ by $\Lambda_{j_1, \ldots ,j_k}$.
Clearly, $\Lambda_{j_1, \ldots ,j_k}$ does not depend on the order of $j_1, \ldots, j_k$.
More generally, for a subset $A \subset \Lambda$, let $A_j:= \overline A \cap \Lambda_j$ and
$A_{j_1, \ldots, j_k} = \overline{A_{j_1, \ldots ,j_{k-1}}} \cap \Lambda_{j_1, \ldots ,j_k}$.
Note that the set $A_{j_1, \ldots, j_k}$ generally depends on the order of $j_1, \ldots, j_k$.
We call $A_{j_1, \ldots, j_k}$ a {\em restriction} of $A$ to $\Lambda_{j_1, \ldots ,j_k}$.
\end{notation}

\begin{definition}\label{def:canonical}
Given the ordered standard simplex $\Delta^m$, the {\em canonical} simplicial map
({\em identification}) between a face $\Delta^m_{j_1, \ldots ,j_k}$ and the ordered standard simplex
$\Delta^{m-k}$ is the simplicial map preserving the order of vertex labels.
\end{definition}

We still adhere to Convention~\ref{con:family_in_simplex}, i.e.,
for a given (ordered) definable simplex $\Lambda$ we write, slightly abusing the notation,
$S_\delta$ instead of $S_\delta \cap \Lambda$, and say that $S_\delta$ is a {\em family in} $\Lambda$.
Similarly, for an upper semi-continuous function $f:\> K \to \Real$, defined on a compact set
$K \subset \Real^n$, and a definable simplex $\Lambda \subset K$ we write $f$ instead of $f|_{\Lambda}$.

\begin{definition}\label{def:ext}
Let $f:\> \Lambda \to \Real$ be an upper semi-continuous definable function on a definable ordered
$m$-simplex $\Lambda$.
Define $f_j:\> (\Lambda \cup \Lambda_j) \to \Real$, where $0 \le j \le m$, as the unique extension,
by semicontinuity, of $f$ to the facet $\Lambda_j$.
Define the function $f_{j_1, \ldots ,j_k}:\> (\Lambda \cup \Lambda_{j_1, \ldots ,j_k}) \to \Real$, where
$(j_1, \ldots ,j_k)$ is a sequence of pair-wise distinct numbers in $\{ 0, \ldots ,m \}$, by induction on $k$,
as the unique extension, by semicontinuity, of $f_{j_1, \ldots ,j_{k-1}}$ to the face $\Lambda_{j_1, \ldots ,j_k}$.
\end{definition}

Consider a monotone family $S_\delta$ in a definable ordered $n$-simplex $\Lambda$.
According to Lemma~\ref{le:S_by_function}, there is a non-negative upper semi-continuous definable function
$f:\>  \Lambda \to \Real$, with $S_\delta=\{\x \in \Lambda|\> f(\x) \ge \delta\}$.
Similarly, for any pair-wise distinct numbers $j_1, \ldots ,j_k$ in $\{ 0, \ldots ,n \}$,
we have
$$(S_\delta)_{j_1, \ldots ,j_k}=\{\x \in \Lambda_{j_1, \ldots ,j_k}|\>
f_{j_1,\ldots,j_k}(\x) \ge \delta\}.$$

\begin{convention}\label{con:labeling}
The vertices of a face $\Lambda_{j_1, \ldots ,j_k}$ of $\Lambda$ inherit the labels of vertices from
$\Lambda$.
When considering a family $(S_\delta)_{j_1, \ldots ,j_k}$ in $\Lambda_{j_1, \ldots ,j_k}$
we rename the vertices so that they have labels $0,1, \ldots ,n-k$ but the order of labels
is the same as it was in $\Lambda$.
\end{convention}

\begin{definition}\label{def:hereditary}
We say that a property of a family $S_\delta$ in $\Lambda$ is {\em hereditary} if it holds true
for any face of $\Lambda$ and any restriction of $S_\delta$ to this face (assuming the
Convention~\ref{con:labeling}).
\end{definition}

Consider the following {\em Basic Conditions} (A)--(D), satisfied for small $\delta >0$.

\begin{enumerate}[(A)]
\item
If $S_\delta \neq \emptyset$ then $(S_\delta)_j \neq \emptyset$ for any $j>0$.
\item
If $(S_\delta)_0= \Lambda_0$ then $S_\delta = \Lambda$.
\item
For every pair $\ell,m$ such that $0 \le \ell <m \le n$ and $(\ell, m) \neq (0,1)$
we have $(S_\delta)_{\ell,m}=(S_\delta)_{m,\ell}$.
\item
Either $S_\delta= \emptyset$ or $\bigcup_\delta  {\rm int} (S_\delta)= \Lambda$.
\end{enumerate}

\begin{lemma}\label{le:restriction_sep}
Let a family $S_\delta$ in a definable $n$-dimensional simplex $\Lambda$ be separable, satisfy the basic
condition (D), and this condition is hereditary.
Then for each $n$-dimensional simplex $\Sigma$ of any triangulation of $\Lambda$ (in particular, of a barycentric
subdivision) the restriction $S_\delta \cap \Sigma$ is separable in $\Sigma$.
\end{lemma}

\begin{proof}
By Remark~\ref{re:separable2}, it is sufficient to prove that for small $\delta >0$ and every face
$\Sigma'$ of $\Sigma$, if $\Sigma' \subset \overline{\Sigma \setminus S_\delta}$ then
$\overline{S_\delta \cap \Sigma} \cap \Sigma'=\emptyset$.
Take $\Sigma'$ such that $\Sigma' \subset \overline{\Sigma \setminus S_\delta}$.
By the basic condition (D), if $\Sigma' \subset \Lambda$ then $\Sigma' \not\subset \overline{\Sigma \setminus S_\delta}$
for small $\delta >0$, hence $\Sigma' \subset \Lambda'$ for some face $\Lambda'$ of $\Lambda$.
Again by (D), $\Lambda' \subset \overline{\Lambda \setminus S_\delta}$, and therefore
$\overline S_\delta \cap \Lambda'=\emptyset$.
It follows that the intersection $\overline{S_\delta \cap \Sigma} \cap \Sigma'=\emptyset$ of smaller sets
is also empty.
\end{proof}

\section{Standard and model families}
\label{sec:standard_and_model}

\begin{convention}
In this section we assume that all monotone definable families satisfy the basic conditions (A)--(D), and these conditions
are hereditary.
\end{convention}

\begin{definition}\label{def:boolean}
Let $S_\delta$ be a monotone family in a definable ordered  $n$-simplex $\Lambda$.
We assign to $S_\delta$ a Boolean function $\psi :\> \{ 0,1 \}^n \to \{ 0,1 \}$ using the
following inductive rule.
\begin{itemize}
\item
If $n=0$ (hence $\Lambda$ is the single vertex 0), then there are two possible
types of $S_\delta$.
If $S_\delta= \Lambda$, then $\psi \equiv 1$, otherwise $S_\delta= \emptyset$ and $\psi \equiv 0$.
\item
If $n>0$, then $\psi|_{x_j=0}$ is assigned to $(S_\delta)_j$ for every $j \neq 0$,
and $\psi|_{x_1=1}$ is assigned to $(S_\delta)_0$
(here vertices of $\Lambda_j$ are renamed $i \to i-1$ for all $i>j$, cf. Definition~\ref{def:canonical}).
\end{itemize}
\end{definition}

\begin{remark}\label{re:bool_funct}
\begin{enumerate}[(i)]
\item
It is obvious that the function $\psi$ is completely defined by its restrictions
$\psi|_{x_1=1}$ and $\psi|_{x_1=0}$, hence by the restrictions $(S_\delta)_0$
and $(S_\delta)_1$.

\item
The basic condition (A) implies that if $\psi (0, \ldots , 0)=0$ then
$\psi \equiv 0$ and $S_\delta= \emptyset$.
The condition (B) implies that if
$\psi(1, \ldots ,1)=1$ then $\psi \equiv 1$ and $S_\delta=\Lambda$.

\item
Because of the basic condition (C), for every pair $0 \le \ell <m \le n$ such that
$\ell <m$ and $(\ell, m) \neq (0,1)$
the restrictions $(S_\delta)_{\ell, m}$ and $(S_\delta)_{m, \ell}$
have the same Boolean function assigned to them.
It follows that under (C) Definition~\ref{def:boolean} is consistent.
It is easy to give an example of a family $S_\delta$ where (C) is not satisfied
and Definition~\ref{def:boolean} becomes contradictory.
\end{enumerate}
\end{remark}

\begin{definition}
A Boolean function $\psi$ is {\em monotone} (decreasing) if replacing 0 by 1 at any position of its
argument (while keeping other positions unchanged) either preserves the value of $\psi$ or
changes it from 1 to 0.
\end{definition}

\begin{lemma}\label{le:boolean_monotone}
The Boolean function $\psi$ assigned to $S_\delta$ is monotone.
\end{lemma}

\begin{proof}
If $S_\delta= \Lambda$, then $\psi \equiv 1$, hence $\psi$ is trivially monotone.
Now assume that $S_\delta \neq \Lambda$, and continue the proof by induction on $n$.
The base of the induction, for $n=0$, is obvious.
Restriction of $\psi$ to $\{ x_j=0 \}$, for any $j$, or to $\{ x_1=1 \}$, is the Boolean function assigned to
a facet of $\Lambda$, which is monotone by the inductive hypothesis.
Hence, if Boolean values are assigned to $n-1$ variables among $x_1,\ldots,x_n$, except on $x_2= \ldots =x_n=1$,
the function $\psi$ is monotone in the remaining variable.
Suppose $\psi$ is not monotone in $x_1$ with $x_2= \cdots =x_n=1$, i.e., $\psi (1,1, \ldots, 1)=1$
and $\psi (0,1, \ldots, 1)=0$.
Then $\psi$ is identically 1 on $\{ x_1=1 \}$, i.e., $(S_\delta)_0= \Lambda_0$,
while $S_\delta \neq \Lambda$.
This contradicts the basic condition (B).
\end{proof}

\begin{definition}
Two monotone families $S_\delta$ and $V_\delta$, in definable simplices $\Sigma$ and
$\Lambda$ respectively, are {\em combinatorially equivalent}
if for every sequence $(j_1, \ldots, j_k)$, where the numbers $j_1, \ldots ,j_k \in
\{ 0, \ldots ,n \}$ are pair-wise distinct, the restrictions
$(S_\delta)_{j_1, \ldots ,j_k} \subset \overline \Sigma$ and
$(V_\delta)_{j_1, \ldots ,j_k} \subset \overline \Lambda$
are simultaneously either empty or non-empty.
\end{definition}

\begin{remark}\label{re:comb_equiv}
Obviously, if two families $S_\delta$ and $V_\delta$ are combinatorially equivalent, then
for every sequence $(j_1, \ldots, j_k)$ of pair-wise distinct numbers the
families $(S_\delta)_{j_1, \ldots , j_k}$ and $(V_\delta)_{j_1, \ldots , j_k}$ are combinatorially
equivalent.
\end{remark}

\begin{remark}\label{re:strongly}
Let $S_\delta$ and $V_\delta$ be two proper monotone families in a 2-simplex $\Lambda$, such that
$\partial_\Lambda S_\delta$ and $\partial_\Lambda V_\delta$ are curve intervals.
Then $S_\delta$ and $V_\delta$ are combinatorially
equivalent if and only if the endpoints of $\partial_\Lambda S_\delta$ can be mapped onto endpoints of
$\partial_\Lambda V_\delta$ so that the corresponding endpoints belong to the same faces of $\Lambda$
for small $\delta >0$.
\end{remark}

\begin{lemma}\label{le:same_function}
Two families $S_\delta$ and $V_\delta$ are assigned the same Boolean function if and only if
these families are combinatorially equivalent.
\end{lemma}

\begin{proof}
Suppose $S_\delta$ and $V_\delta$ are assigned the same Boolean function.
According to Definition~\ref{def:boolean}, for any $(j_1, \ldots, j_k)$ the restrictions
$(S_\delta)_{j_1, \ldots , j_k}$ and $(V_\delta)_{j_1, \ldots , j_k}$ are also assigned the same
Boolean function (after renaming the vertices whenever appropriate).
In particular, this function is identical 0 or not identical 0 simultaneously for both faces.

We prove the converse statement by induction on $n$, the base case $n=0$, being obvious.
According to Remark~\ref{re:comb_equiv}, for every $j$ the restrictions $(S_\delta)_j$ and $(V_\delta)_j$
are combinatorially equivalent.
By the inductive hypothesis these restrictions are assigned the same $(n-1)$-variate Boolean function.
According to Definition~\ref{def:boolean}, for every $j$, the restrictions to $x_j=0$ of
Boolean functions $\psi$ and $\varphi$, assigned to $S_\delta$ and $V_\delta$
respectively coincide, and also restrictions of $\psi$ and $\varphi$ to $x_1=1$ coincide.
Hence, $\psi = \varphi$.
\end{proof}

Observe that when $n=1$ (respectively, $n=2$) there are exactly three (respectively, six) distinct monotone
Boolean functions.
Therefore, Lemma~\ref{le:same_function} implies that in this case there are exactly three
(respectively, six) distinct combinatorial equivalence classes of monotone families.

\begin{definition}
\label{def:lex-monotone}
A Boolean function $\psi:\> \{ 0,1 \}^n \to \{ 0,1 \}$ is {\em lex-monotone} if it is monotone with
respect to the lexicographic order of its arguments, assuming $x_1 < \cdots < x_n$.
\end{definition}

Note that when $n=1$ all monotone functions are lex-monotone, whereas for $n=2$ all functions except one
are lex-monotone.
In general, for the number of all monotone Boolean functions ({\em Dedekind number}) no closed-form expression is
known at the moment of writing.
On the other hand, the number of lex-monotone functions is easy to obtain.

\begin{lemma}\label{le:number_of_lex-monotone}
\begin{enumerate}[(i)]
\item
There are $2^n+1$ lex-monotone functions $\psi:\> \{ 0,1 \}^n \to \{ 0,1 \}$.
\item
If $\psi$ is lex-monotone then its restriction $\psi|_{x_i= \alpha}$, for
any $i= 1, \ldots ,n$ and any $\alpha \in \{ 0,1 \}$, is a lex-monotone function in $n-1$
variables.
\end{enumerate}
\end{lemma}

\begin{proof}
\noindent (i)\ Every monotone function $\psi:\> A \to \{ 0,1 \}$, where
$A$ is a totally ordered finite set of cardinality $k$, can be represented as a $k$-sequence
of the kind $\psi=(1, \ldots ,1,0, \ldots ,0)$.
There are $k+1$ such sequences.
In our case $k=2^n$.
\medskip

\noindent (ii)\ Straightforward.
\end{proof}

\begin{definition}\label{def:standard}
A combinatorial equivalence class of a family $S_\delta$ is called {\em combinatorially standard}, and
the family itself is called {\em combinatorially standard}, if the corresponding Boolean
function is lex-monotone.
\end{definition}

The following picture shows representatives of all combinatorially standard families in cases $n=1$ and $n=2$.

\begin{figure}[hbt]
\vspace*{-0.75in}
       \centerline{
          \scalebox{0.6}{
             \includegraphics{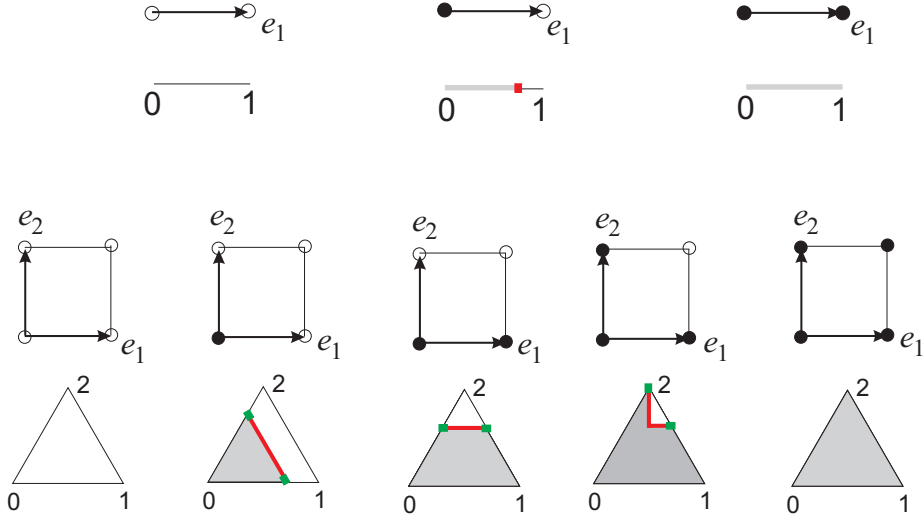}
             }
           }
\vspace*{-1.5in}
\caption{Standard families in dimensions one and two, with corresponding lex-monotone functions.}
\label{fig:1D-2D}
\end{figure}

\begin{lemma}\label{le:codim}
If a family $S_\delta$ is combinatorially standard, then the family $(S_\delta)_{j_1, \ldots , j_k}$
for each $(j_1, \ldots , j_k)$ is combinatorially standard.
\end{lemma}

\begin{proof}
Follows immediately from the definition of the combinatorial equivalence.
\end{proof}

\begin{definition}\label{def:model}
Let $\psi (x_1, \ldots ,x_n)$ be a lex-monotone Boolean function.
A family $V_\delta$ in the ordered standard simplex $\Delta^n$ is called the {\em model family}
assigned to $\psi$, if it is constructed inductively as follows.
\begin{enumerate}
\item
The {\em non-proper} family $V_\delta=\emptyset$ is assigned to $\psi \equiv 1$ when $\delta >1$, and
to $\psi \equiv 0$.
The {\em non-proper} family $V_\delta=\Delta^n$ is assigned to $\psi \equiv 1$, when $\delta \le 1$.
In case $n=0$, there are no other families.

\item
If $\psi\not\equiv 0$ while $\psi|_{x_1=1}\equiv 0$, then let
$V_\delta=\Delta^n \cap \{t_1+ \cdots +t_n \le 1- \delta \}$.

\item
If $\psi|_{x_1=0}=\psi|_{x_1=1}$, then $V_\delta$ is the pre-image of the model family in
$\Delta^{n-1}$ assigned to $\psi|_{x_1=0}$ under the projection map $\Delta^n \to \Delta^{n-1}$ gluing
together vertices $0$ and $1$ of $\Delta^n$.

\item
Let $\psi|_{x_1=0}\ne \psi|_{x_1=1} \not\equiv 0$.
Define $U_\delta$ in $\Delta^n$ as the family
assigned as in (3), to the Boolean function $\varphi (x_1, \ldots, x_n)$ such that
$\varphi|_{x_1=0}=\varphi|_{x_1=1}=\psi|_{x_1=0}$.
Define $W_{\delta}$ in $\Delta^n$ as the family
assigned as in (3), to the Boolean function $\xi (x_1, \ldots ,x_n)$ such that
$\xi|_{x_1=0}=\xi|_{x_1=1}=\psi|_{x_1=1}$.
Let $\widehat \Delta$ denote the closure of $\Delta^n \cap \{ 2t_1+t_2 \cdots +t_n < 1 \}$.
Define the family $V_\delta$ in $\Delta^n$ as $(U_\delta \cap \widehat \Delta) \cup W_\delta$.
\end{enumerate}
\end{definition}

Here are the lists of all {\em model families} in ordered $\Delta^n$ in cases $n=1$ and $n=2$.

{\bf Case} $n=1$.

\begin{itemize}
\item[(0)]
$\emptyset$
\item[(1)]
$\{t \le 1-\delta \} \cap \Delta^1$
\item[(2)]
$(0,1)= \Delta^1$
\end{itemize}

{\bf Case} $n=2$.
\begin{itemize}
\item[(0)]
$\emptyset$
\item[(1)]
$\{ t_1+t_2 \le 1- \delta \} \cap \Delta^2$
\item[(2)]
$\{ t_2 \le 1-\delta \} \cap \Delta^2$
\item[(3)]
$ (\{t_2\le 1-\delta\}\cup\{2 t_1+t_2\le 1\}) \cap \Delta^2$
\item[(4)]
$\Delta^2$
\end{itemize}

In cases $n \le 2$, the Figure~\ref{fig:1D-2D} actually shows all model families.

In the case $n=1$, all monotone Boolean functions are lex-monotone, hence families (0)--(2) represent
all combinatorial classes of families satisfying the basic conditions, and these classes are
combinatorially standard.

In the case $n=2$, there is only one monotone function which is not lex-monotone,
$\psi|_{x_1=0} \equiv 1, \> \psi|_{x_1=1} \equiv 0$, and we can define the {\em non-standard
model family} for $\psi$ as

\begin{itemize}
\item[(5)]
$(\{t_1+t_2\le 1-\delta\}\cup\{2 t_1+t_2\le 1\})\cap\Delta^2$.
\end{itemize}

\begin{remark}
The non-standard model family (5) can be obtained as a result of the procedure in item (4) of
Definition~\ref{def:model}.
Also (5) is combinatorially equivalent to $\{ t_1/(1-\delta) +t_2 \le 1 \} \cap \Delta^2$
which corresponds to the standard blow-up at the vertex labeled by 2 of $\Delta^2$.
\end{remark}

Note that for large $n$ most of the monotone Boolean functions are not lex-monotone (Dedekind number grows
superexponentially), hence most families are not standard.

\begin{lemma}\label{le:family_to_function}
There is a bijection between all combinatorially standard equivalence classes and all lex-monotone
Boolean functions.
\end{lemma}

\begin{proof}
According to Lemma~\ref{le:same_function}, to any two combinatorially equivalent families the
same Boolean function is assigned.
It is straightforward to show that the model family, corresponding by Definition~\ref{def:model}
to a given lex-monotone function,
is assigned this same function by Definition~\ref{def:boolean}.
By Lemma~\ref{le:same_function}, any two families having the same Boolean function belong to the same
equivalence class.
\end{proof}

\begin{definition}
A lex-monotone Boolean function $\psi (x_1, \ldots, x_n)$ is called {\em separable} if
the set $\{(x_1, \ldots, x_n)|\> \psi (x_1, \ldots, x_n)=1 \}$ consists of either 0 or $2^k$ points,
for some $k \in \{ 0, \ldots ,n \}$.
\end{definition}

\begin{remark}\label{re:number_of_separable}
The number of separable functions in $n$ variables is $n+2$, since for each
$k \in \{0, \ldots , 2^n \}$ there is the unique lex-monotone function $\psi$ with
cardinality of $\{ \psi=1 \}$ equal to $k$ (cf. the proof of Lemma~\ref{le:number_of_lex-monotone}, (i)).
It follows that all lex-monotone
functions are separable for $n \le 1$ and there is a single non-separable lex-monotone function for $n=2$.
\end{remark}

\begin{lemma}\label{le:separableBoolean}
A lex-monotone Boolean function $\psi (x_1, \ldots, x_n)$, where $n \ge 0$, is separable if and
only if either $\psi|_{x_1=1} \equiv 0$, or $\psi|_{x_1=0}$ and $\psi|_{x_1=1}$ are separable and equal.
\end{lemma}

\begin{proof}
Let $\psi|_{x_1=1} \equiv 0$, or $\psi|_{x_1=0}$ and $\psi|_{x_1=1}$ be separable and equal.
By Remark~\ref{re:number_of_separable}, $\psi$ is separable when $n=0$.
If $\psi|_{x_1=1} \equiv 0$ then, by lex-monotonicity, $\{\psi|_{x_1=0}=1 \}$ consists of a single point.
If $\psi|_{x_1=1} \not\equiv 0$ then the cardinality of
$\{\psi =1 \}$ is twice the cardinality of $\{\psi |_{x_1=0}=1 \}$.
The latter, by the assumption, is a power of 2.

Conversely, there are exactly $n+2$ such distinct functions $\psi$, while the number of different separable
functions is also $n+2$, by Remark~\ref{re:number_of_separable}.
\end{proof}

\begin{lemma}
Let $V_\delta=\{ f \ge \delta \} \cap \Delta$ be a model family, and $\psi$ its corresponding lex-monotone
Boolean function.
The following properties are equivalent.
\begin{enumerate}[(i)]
\item
$V_\delta$ is separable;
\item
$\psi$ is separable;
\item
$f$ is continuous in $\overline \Delta$.
\end{enumerate}
\end{lemma}

\begin{proof}
We prove the lemma by induction on $n$, the basis for $n=0$ being obvious.

Let $\psi$ be separable, and $n>0$.
By Lemma~\ref{le:separableBoolean}, either $\psi|_{x_1=1} \equiv 0$, or
$\psi|_{x_1=0}$ and $\psi|_{x_1=1}$ are separable and equal.
In the first case, by Definition~\ref{def:model}, (2), the family
$V_\delta$ coincides with $\Delta^n \cap \{t_1+ \cdots +t_n \le 1- \delta \}$, and hence is separable.
It is clear that the function $f$ for this family is continuous.

In the second case, by Definition~\ref{def:model}, (3),
$V_\delta$ is the pre-image of the separable (by the inductive hypothesis) family in
$\Delta^{n-1}$, assigned to the separable function $\psi|_{x_1=0}$ and having a continuous
defining function, under the projection map $\Delta^n \to \Delta^{n-1}$ gluing together vertices
$0$ and $1$ of $\Delta^n$, and hence is separable.
It is clear that the the function $f$ for $V_\delta$ is also continuous.

Let $\psi$ be not separable.
Then, by Lemma~\ref{le:separableBoolean}, there are two possibilities.
First, $\psi|_{x_1=1} \not\equiv 0$, and at least one of the restrictions $\psi|_{x_1=0}$ or
$\psi|_{x_1=1}$ is not separable.
Let, for definiteness, it be $\psi|_{x_1=0}$, and let $\Delta'$ be a face of $\Delta^{n-1}_1$ such that
$\Delta' \subset \overline{\Delta^{n-1}_1 \setminus (V_\delta)_1}$.
Note that $\Delta'$ is also a face of $\Delta^n$, and
$\Delta' \subset \overline{\Delta^{n-1} \setminus V_\delta}$.
By the inductive hypothesis, $(\overline V_\delta)_1 \cap \Delta' \neq \emptyset$, hence
$\overline V_\delta \cap \Delta' \neq \emptyset$, and we conclude that $V_\delta$ is not separable.
Also by the inductive hypothesis, the defining function $f_1$ of the family $(V_\delta)_1$
(see Definition~\ref{def:ext}) is not continuous, hence the function $f$ is not continuous.

The second possibility is that both restrictions, $\psi|_{x_1=0}$ and $\psi|_{x_1=1}$, are separable
but different functions.
Then the faces $\Delta'$ and $\Delta''$ of $\Delta^n$ of the largest dimensions such that
$\Delta' \subset \overline{\Delta^n_0 \setminus (V_\delta)_0}$ and
$\Delta'' \subset \overline{\Delta^n_1 \setminus (V_\delta)_1}$ are also different, one of them is a face
of another (say, $\Delta'' \subset \overline{\Delta'}$), and both lie in
$\overline{\Delta^n \setminus V_\delta}$.
It follows that $\overline V_\delta \cap \Delta' \neq \emptyset$, hence $V_\delta$ is not separable.
The values of the function $f$ at points of $\Delta^n_0$ sufficiently close to $\Delta'$ and sufficiently
far from $\Delta''$, are close to 0, while the values of $f$ at points of $\Delta^n_1$ with the same property
are separated from 0.
It follows that $f$ is not continuous.
\end{proof}

\begin{lemma}\label{le:model_regular}
For each model family $V_\delta$ in the ordered standard simplex $\Delta^n$
its interior ${\rm int}(V_\delta)$ is a semi-monotone set.
\end{lemma}

\begin{proof}
Consider cases (1)--(4) of Definition~\ref{def:model}.

If $V_\delta$ is defined according to either (1) or (2), then ${\rm int}(V_\delta)$ is
a convex set, hence semi-monotone.

For other cases
we prove by induction on $n$, that the family ${\rm int}(V_\delta)$ satisfies condition (ii) in
Lemma~\ref{le:cond_for_semi}, this implies its semi-monotonicity.
The base for $n=0$ is obvious.

Suppose $V_\delta$ is defined according to case (3), and let $V'_\delta$ be the family in
$\Delta^{n-1}$ assigned to $\psi|_{x_1=0}$ under the projection $\Delta^n \to \Delta^{n-1}$.
Take any $\x=(\x',x_1) \in {\rm int}(V_\delta)$ and $j \in \{2, \ldots ,n\}$,
then, by the inductive hypothesis, the interval
$$I'_{\x',j}:= \{(y_2, \ldots ,y_n) \in \Real^{n-1}|\> 0<y_j<x_j,\> y_i=x_i\> \text{for}\> i\ne j\}$$
lies in ${\rm int}(V'_\delta)$.
Since $\Delta^n$ is convex, it follows that $(I'_{\x',j} \times {\rm span} \{ x_1 \}) \cap \Delta^n$
lies in ${\rm int}(V_\delta)$, therefore $I_{\x,j} \subset {\rm int}(V_\delta)$ and
$I_{\x, 1} \subset {\rm int}(V_\delta)$, i.e., ${\rm int}(V_\delta)$
satisfies the condition (ii) in Lemma~\ref{le:cond_for_semi}.

Now suppose that $V_\delta$ is defined according to case (4).
By the same argument as was applied to $V_\delta$ in case (3), we show that families ${\rm int}(U_\delta)$ and
${\rm int}(W_\delta)$ satisfy the condition (ii) in Lemma~\ref{le:cond_for_semi}.
Simplex $\widehat \Delta$ satisfies this condition trivially.
Then Lemma~\ref{le:cond_for_semi} implies that the set
$${\rm int}(V_\delta)=({\rm int}(U_{\delta}) \cap {\rm int} (\widehat \Delta)) \cup {\rm int}(W_{\delta})$$
also satisfies the condition (ii) in Lemma~\ref{le:cond_for_semi}.

Using Lemma~\ref{le:cond_for_semi}, we conclude that in all cases, ${\rm int}(V_\delta)$ is semi-monotone.
\end{proof}

\begin{remark}\label{re:model_boundary}
For a proper model family $V_\delta$ its boundary $\partial_{\Delta^n} V_\delta$ in $\Delta^n$
is not necessarily a monotone cell.
However, $\partial_{\Delta^n} V_\delta$ is always a regular $(n-1)$-cell, because its complement in
the whole boundary $\partial V_\delta$ is a union of monotone (hence, regular) $(n-1)$-cells
glued together with the same nerve as that of the complement of a vertex in the boundary of
the simplex $\Delta^n$ (considered as a simplicial complex of all of its faces).
An exception is the family with $\psi|_{x_1=1}\equiv 0$, for which the nerve is the same as for the
complement of an $(n-1)$-face instead of a vertex.
\end{remark}

\begin{lemma}\label{le:model_prop}
Every model family $V_\delta$ in the standard ordered simplex $\Delta^n$ satisfies the
following properties.
\begin{enumerate}
\item
$V_\delta$ satisfies the basic conditions (A)--(D), and these conditions are hereditary.
\item
For any $(j_1, \ldots, j_k)$ the restriction
$(V_\delta)_{j_1, \ldots , j_k}$
is a model family.
\item
${\rm int}(V_\delta)$ is a regular $n$-cell while $\partial_\Delta V_\delta$ is a regular $(n-1)$-cell.
\end{enumerate}
\end{lemma}

\begin{proof}
Properties (1) and (2) follow immediately from Definition~\ref{def:model}.

The property (3) follows from Lemma~\ref{le:model_regular}, the fact that a semi-monotone set is a regular cell,
and Remark~\ref{re:model_boundary}.
\end{proof}

\begin{lemma}\label{le:model_bary}
Let $V_\delta$ be a model family in the standard ordered simplex $\Delta^n$.
Then for each simplex $\Sigma$ of the barycentric subdivision of $\Delta^n$ the restriction $V_\delta \cap \Sigma$
is a separable family.
\end{lemma}

\begin{proof}
We first describe, by induction on $n$, a triangulation of $\Delta$ (which is coarser than the barycentric
subdivision) such that the restriction of $V_\delta$ to each its simplices is separable.

If $n=0$ then the family is already separable.

If $V_\delta$ is defined according to either (1) or (2) of Definition~\ref{def:model}, then
it is already separable.

Suppose $V_\delta$ is defined according to case (3), and let $V'_\delta$ be the family in
$\Delta^{n-1}$ assigned to $\psi|_{x_1=0}$ under the projection $\Delta^n \to \Delta^{n-1}$.
By the inductive hypothesis, $V'_\delta$ can be partitioned into separable families.
The pre-images of these families form a partition of $V_\delta$ into separable families.

Now suppose that $V_\delta$ is defined according to case (4).
Partition $\Delta^n$ into two $n$-simplices,
$\Delta^n_<:= \Delta^n \cap \{ 2t_1+t_2 \cdots +t_n < 1 \}$ and
$\Delta^n_>:= \Delta^n \cap \{ 2t_1+t_2 \cdots +t_n >1 \}$, and one $(n-1)$-simplex
$\Delta^{n-1}_= := \Delta^n \cap \{ 2t_1+t_2 \cdots +t_n =1 \}$,
where the vertex $(1/2, 0, \ldots ,0)$ on the edge $\Delta^n_{2, \ldots,n}$ has label 1 in $\Delta^n_<$
and label 0 in $\Delta^n_>$.

Define $U'_{\delta}:= U_\delta \cap \Delta^n_<$ and $W'_{\delta}:= W_\delta \cap \Delta^n_<$.
Families $U'_\delta$ and $W'_\delta$ are partitioned into separable families as in the case (3),
while the $(n-1)$-dimensional model family
$(\overline U'_\delta \cup \overline W'_\delta) \cap \Delta^{n-1}_=$
(it corresponds to the Boolean function $\psi|_{x_1=0}$) is partitioned into separable families, according
to the inductive hypothesis.

There is a refinement of the described triangulation of $\Delta^n$ which is the barycentric subdivision of
$\Delta^n$.
By Lemma~\ref{le:restriction_sep}, the restriction of $V_\delta$ to each simplex of this barycentric subdivision
is separable.
\end{proof}

\section{Topological equivalence}
\label{sec:top-equivalence}

\begin{definition}\label{def:equiv}
Consider two monotone families $S_\delta$ and $V_\delta$ in definable ordered $m$-simplices
$(\Sigma, \Phi)$ and $(\Lambda, \Psi)$ in $\Real^n$ respectively, where
$\Phi:\> \overline \Delta \to \overline \Sigma$, $\Psi:\> \overline \Delta \to \overline \Lambda$ are
homeomorphisms, and $\overline \Delta$ is the closure of the standard ordered $m$-simplex in $\Real^m$
(see Definition~\ref{def:simplex}).

Families $S_\delta$ and $V_\delta$ are {\em topologically equivalent} if
there exist two face-preserving homeomorphisms
$h_1:\> \overline \Lambda \to \overline \Sigma$ and $h_2:\> \overline \Sigma \to \overline \Lambda$,
not depending on $\delta$, such that for small $\delta>0$ the inclusions $S_\delta \subset h_1(V_\delta)$ and
$V_\delta \subset h_2(S_\delta)$ are satisfied.

Families $S_\delta$ and $V_\delta$ are {\em strongly topologically equivalent} if they are topologically equivalent, and
for small $\delta >0$ there is a face-preserving homeomorphism $h_\delta:\> \overline \Sigma \to \overline \Lambda$
such that $h_\delta(S_\delta)=V_\delta$.
\end{definition}

\begin{remark}
It is clear that if two families $S_\delta$ and $V_\delta$ are topologically equivalent with homeomorphisms
$h_1,\ h_2$, then there exist two face-preserving homeomorphisms, namely,
$$h':=\Phi^{-1} \circ h_1 \circ \Psi,\>
h'':= \Psi^{-1} \circ h_2 \circ \Phi:\> \overline \Delta \to \overline \Delta,$$
with the property that
for small $\delta>0$ the inclusions $\Phi^{-1}(S_\delta) \subset h'(\Psi^{-1}(V_\delta))$ and
$\Psi^{-1}(V_\delta) \subset h''(\Phi^{-1}(S_\delta))$ are satisfied.
Conversely, given two homeomorphisms $h',\> h'':\> \overline \Delta \to \overline \Delta$,
satisfying these inclusion properties, there are homeomorphisms
$$h_1:=\Phi \circ h' \circ \Psi^{-1}:\> \overline \Lambda \to \overline \Sigma\quad
\text{and}\quad h_2:= \Psi \circ h'' \circ \Phi^{-1}:\> \overline \Sigma \to \overline \Lambda,$$
realizing the topological equivalence of $S_\delta$ and $V_\delta$.
\end{remark}

Introduce the following new basic condition on a monotone family $S_\delta$ in a definable ordered $m$-dimensional
simplex $\Lambda$.
\begin{itemize}
\item[(E)]
The interior ${\rm int} (S_\delta)$ is either empty or a regular $m$-cell and
the boundary $\partial_\Lambda S_\delta$ in $\Lambda$ is either empty or a regular $(m-1)$-cell.
\end{itemize}

\begin{convention}
In this section we assume that all monotone definable families satisfy the basic conditions (A)--(E), and these
conditions are hereditary.
\end{convention}

\begin{remark}\label{re:(F)}
If $\overline S_\delta$ contains a facet $\Lambda_j$, then $S_\delta$ contains a neighborhood of $\Lambda_j$ in
$\Lambda$ for all small positive $\delta$.
If, in addition, $\overline{(S_\delta)}_i$ contains a face $\Lambda_{i,j}$ for every $i \neq j$, then
$\overline S_\delta$ contains the neighborhood in $\overline{\Lambda}$ of the closed facet
$\overline {\Lambda_j}$.
\end{remark}

\begin{lemma}\label{le:limit}
Given a proper family $S_\delta$ in a definable $m$-dimensional simplex $\Lambda$,
the Hausdorff limit $L$ of $\partial_\Lambda S_\delta$, as $\delta \searrow 0$,
is the closure of a face $\Lambda_{0,1, \ldots , r}$ of $\Lambda$ for some $r \in \{0,1, \ldots ,m-1 \}$.
\end{lemma}

\begin{proof}
\noindent (1)\quad We first prove that $L \subset \partial \Lambda$.
Let, contrary to claim, $L \cap \Lambda \neq \emptyset$.
Then, by the basic condition (D), there exists $\x \in \Lambda$ such that $\x \in L \cap {\rm int}(S_\delta)$ for all sufficiently
small $\delta$, which is a contradiction.
\medskip

\noindent (2)\quad Next we prove that $L \subset \overline{\Lambda}_0$.
Suppose that $\x \in L$ but $\x \not\in \overline{\Lambda}_0$.
Since, by (1), $L \subset \partial \Lambda$, the point $\x$ lies in a face $\Lambda'$ of $\Lambda$ such that
$\Lambda' \cap \overline \Lambda_0= \emptyset$.
By the hereditary basic condition (D) applied to $\Lambda'$, we have that $\x \in \overline S_\delta \cap \Lambda'$ for small $\delta$.
Then in the neighbourhood of $\x$ in $\Lambda$ there is a point $\y$ such that $\y \in \Lambda \setminus S_\delta$
for any small $\delta$.
This contradicts to the basic condition (D).
\medskip

\noindent (3)\quad We now prove that $\partial_{\Lambda_0}(S_\delta)_0 = \partial(\partial_\Lambda S_\delta) \cap \Lambda_0$.
By the basic condition (E), $\partial(\partial_\Lambda S_\delta)$ is an $(m-2)$-sphere in $\partial \Lambda$.
By Newman's theorem (3.13 in \cite{RS}), it is the boundary of the closed $(m-1)$-ball $\overline S_\delta \cap \partial \Lambda$.
Hence $\partial(\partial_\Lambda S_\delta) \cap \Lambda_0$ is the boundary in $\Lambda_0$ of
$\overline S_\delta \cap \Lambda_0=(S_\delta)_0$ which we conventionally denote by $\partial_{\Lambda_0}(S_\delta)_0$.
\medskip

\noindent (4)\quad Now we are able to finalise the proof the lemma by induction on $m$, the base being trivial.
If $L= \overline \Lambda_0$ we are done.
Otherwise, let $L_0$ be the Hausdorff limit of $\partial_{\Lambda_0} (S_\delta)_0$, as $\delta \searrow 0$,
and we prove that $L=L_0$.

We have $L_0 \subset L$, because, by (3),
$\partial_{\Lambda_0}(S_\delta)_0 \subset \partial(\partial_\Lambda S_\delta) \subset \overline{ \partial_\Lambda S_\delta}$
for small $\delta$, and we can pass to Hausdorff limit
in these inclusions.
Conversely, there does not exist $\x \in L$ such that $\x \not\in L_0$, because, by the hereditary basic condition (D),
the latter implies that $\x \in \overline{(S_\delta)}_0$ for all small $\delta$, hence $\x \not\in L$.

By the inductive hypothesis, $L_0$ is a closure of a face $\Lambda_{0,1, \ldots , r}$ of $\Lambda$ for some
$r \in \{1, \ldots ,m-1 \}$, hence the same is true for $L$.
\end{proof}

\begin{lemma}\label{le:mutual_separable}
Let $S_\delta$ be a separable family in the standard simplex $\Delta$,
and $V_\delta$ the combinatorially equivalent model family in $\Delta$.
Then for any $\delta>0$ there exists $\eta>0$ such that $S_\delta \subset V_\eta$ and $V_\delta \subset S_\eta$.
\end{lemma}

\begin{proof}
For non-proper families the statement is trivial.
Let $S_\delta$ be proper.
Let $\Delta'$ be the face of $\Delta$ such that its closure is the Hausdorff limit of
$\partial_\Delta S_\delta$ as $\delta \searrow 0$ (such $\Delta'$ exists by Lemma~\ref{le:limit}).
Then $\Delta'$ is the unique face of $\Delta$ of the maximal dimension such that
$\Delta' \subset \overline{\Delta \setminus S_\delta}$.
Let
\begin{eqnarray*}
\min(S_\delta) &=& \min_{\x \in S_\delta} \mathrm{dist}(\x,\overline{\Delta'}), \\
\max(S_\delta) &=& \max_{\x \in S_\delta} \mathrm{dist}(\x,\overline{\Delta'}), \\
\min(V_\delta) &=& \min_{\x \in S_\delta} \mathrm{dist}(\x,\overline{\Delta'}), \\
\max(V_\delta) &=& \max_{\x \in S_\delta} \mathrm{dist}(\x,\overline{\Delta'}).
\end{eqnarray*}
All functions $\max( \cdot),\> \min(\cdot)$ are monotone decreasing with $\delta \searrow 0$.
The separability of both families implies that the minima are positive.
For a given $\delta$, choosing $\eta$ so that $\max(V_\eta) < \min(S_\delta)$, we get $S_\delta \subset V_\eta$,
while choosing $\eta$ so that $\max(S_\eta) < \min(V_\delta)$, we get $V_\delta \subset S_\eta$.
\end{proof}

\begin{definition}\label{def:standard_family}
A combinatorial equivalence class of a family $S_\delta$ (satisfying the basic conditions (A)--(E))
is called {\em standard}, and the family itself is called {\em standard}, if the corresponding
Boolean function is lex-monotone.
\end{definition}

\begin{remark}\label{re:standard}
A {\em combinatorially standard} family, from Definition~\ref{def:standard},
does not need to satisfy the condition (E).
Consider, for example, $S_\delta= \Lambda^2 \setminus B_\delta$, where $B_\delta \subset \Lambda^2$ is
homeomorphic to an open disk, with $\overline{B_\delta}$ intersecting the boundary of $\Lambda^2$ at a point
$\x \neq \Lambda^2_{1,2}$, such that $B_\delta$ contracts to $\x$ as $\delta \searrow 0$.
The family $S_\delta$ satisfies (A)--(D), is combinatorially standard (combinatorially equivalent to
$V_\delta= \Lambda^2$) but not standard.
\end{remark}

\begin{lemma}\label{le:equivalencies}
Let $\{ T^1_\delta, \ldots T^r_\delta \}$ be a finite set of standard monotone families in the standard ordered
$1$-simplex $\Delta^1$.
There is a face-preserving homeomorphism $h^1:\> \overline{\Delta^1} \to \overline{\Delta^1}$
not depending on $\delta$, satisfying the following property.
For every $i=1, \ldots ,r$ there exists one of the model families, $W^i_\delta$,
(among types (1), (2), (3)) of the same combinatorial type as $T_\delta^i$, such that for small $\delta >0$
the inclusions $T^i_\delta \subset h^1(W^i_\delta)$ and $W^i_\delta \subset h^1(T^i_\delta)$ hold.
\end{lemma}

\begin{proof}
Observe that for non-proper families any homeomorphism $h^1$ is suitable.
The basic conditions (A)--(E) imply that every proper family $T_\delta^i$ coincides with $(0,u_i(\delta)]$
for some functions $u_i(\delta):(0,1) \to (0,1)$, such that $\lim_{\delta \to 0} u_i(\delta) = 1$
(in particular, functions $u_i$ are monotone decreasing for small positive $\delta$).
In this case homeomorphism $h^1$ can be defined by any monotone function $h^1(t)$ satisfying the conditions:
\begin{itemize}
\item
$h^1(0)=0$,
\item
$h^1(1)=1$,
\item
$h^1(u_i(\delta))>1-\delta$ and $h^1(1-\delta)>u_i(\delta)$ for each $i$ such that $T^i_\delta$ is proper, and for
small $\delta>0$.
\end{itemize}
To achieve the last property, the graph of $h^1|_{(0,1)}$ should be situated above parametric curves
$(u_i(\delta), 1- \delta),\> (1- \delta, u_i(\delta)) \subset (0,1)^2$ for small $\delta >0$.
\end{proof}

\begin{theorem}\label{th:equivalencies}
Let $\{ T^1_\delta, \ldots ,T^r_\delta \}$ and $\{ S^1_\delta, \ldots , S^k_\delta \}$ be finite sets of
standard monotone families in the standard ordered $1$-simplex $\Delta^1$ and $2$-simplex
$\Delta^2$, respectively.
There exist face-preserving homeomorphisms $h^1:\> \overline{\Delta^1} \to \overline{\Delta^1}$ and
$h^2:\> \overline{\Delta^2} \to \overline{\Delta^2}$ not depending on $\delta$, satisfying the
following properties.
The homeomorphism $h^1$ satisfies Lemma~\ref{le:equivalencies}.
The restriction of $h^2$ to closures of all facets of $\Delta^2$ coincide with $h^1$,
after canonical identification of each facet with the standard ordered $1$-simplex
(see Definition~\ref{def:canonical}).
For every $i=1, \ldots ,k$ there exists a model family,
$V^i_\delta$, such that for every small $\delta >0$
the inclusions $S^i_\delta \subset h^2(V^i_\delta)$ and $V^i_\delta \subset h^2(S^i_\delta)$)
hold.
\end{theorem}

\begin{proof}
We may assume all families $S^i_\delta$ to be proper (of type (1), (2) or (3)).

We first define the homeomorphism $h^1$.
It will satisfy Lemma~\ref{le:equivalencies} for families $T^1_\delta, \ldots ,T^r_\delta$  and three other
special families in $\Delta^1$.
These special families are needed for $h^1$ to become the restriction to edges of $\Delta^2$ of the
homeomorphism $h^2$, required in the theorem.

To construct the special families, we introduce four auxiliary functions
$s(\delta)$, $a(\delta)$, $b(\delta)$, $c(\delta)$.
Each of these functions is defined and continuous for small $\delta>0$, has values in $(0,1)$ and tends to 1
as $\delta \to 0$.
In particular, each function is monotone decreasing for small positive $\delta$.

Let $s(\delta):= \max_i \{s_i(\delta) \}$, where $s_i(\delta)$ is the boundary of the monotone family $S_\delta^i$
of the type (2) or (3), restricted to $\Delta_0^2$
(under the canonical affine map identifying this edge with $\Delta^1$).

Let $a(\delta)$ be a function
such that for all families $S_\delta^i$ with combinatorial types either (2) or (3), the inclusion
$\{t_2\le a(\delta) \} \cap \Delta^2 \subset S_\delta^i$ holds for all small $\delta >0$.
Such a function exists, since by Remark~\ref{re:(F)}, the closure
$\overline{S_\delta^i}$ contains the neighborhood in $\overline{\Delta^2}$ of the closed edge
$\overline {\Delta_2^2}$.

Let $b(\delta)$ be a function, such that $s(\delta)<b(\delta)<1$.
Then for any $S_\delta^i$ the difference $(\{t_2=b(\delta)\} \cap \overline{\Delta^2}) \setminus S_\delta^i$ contains
a neighborhood in the segment $\{t_2=b(\delta)\} \cap \overline{ \Delta^2}$ of its
end $(1-b(\delta),b(\delta))$ on the edge $\Delta_0^2$ of $\Delta^2$.

Let $S^i_\delta$ be of the type either (1) or (2), and $V_\delta^i$ the combinatorially equivalent model family.
According to Lemma~\ref{le:mutual_separable}, for any $\delta>0$ there exists $\eta(\delta)>0$ such that
$S_\delta \subset V_{\eta(\delta)}$ and $V_\delta \subset S_{\eta(\delta)}$.
Let $c_i(\delta)$ be a monotone continuous function defined for all small positive $\delta$
so that $c_i(\delta)>1-\eta(\delta)$.
Define the function $c(\delta)$ as $c(\delta):= \max_i \{ c_i(\delta) \}$.

We construct the homeomorphism $h^1$, as in Lemma~\ref{le:equivalencies}, for
$T^1_\delta, \ldots ,T^r_\delta$ and the following three
families: $T_\delta^{r+1}=(0,b(\delta)]$, $T_\delta^{r+2}=(0,1-a^{-1}(1-\delta)]$, and
$T_\delta^{r+3}=(0, c(\delta)]$, all defined on $\Delta^1$.
It follows, in particular, that
$h^1(1-\delta)>b(\delta)$, $h^1(a(\delta))>1-\delta$, and $h^1(1- \delta) >c(\delta)$ for all small $\delta>0$.
We will also assume, without loss of generality, that $h^1(t)>t$ for all $t \in \Delta^1$.

To describe the homeomorphism $h^2$, we need two more auxiliary monotone decreasing functions,
$p(t)$ and $q(t)$.

Let $p(t)$ be a function on $[0,1]$ satisfying the following properties.
\begin{enumerate}[(i)]
\item
$0<p(t)/(1-t) < 1/2$ for all $t \in (0,1)$, and $p(0)=1/2$;
\item
$p(t)/(1-t)$ is decreasing as a function in $t$;
\item
there exists $\delta_0>0$ such that,
for each $S_\delta^i$ with combinatorial type (3), the inclusion
$(\{t_1=p(t_2)\} \cap \Delta^2) \subset S_{\delta_0}^i$ holds.
\end{enumerate}
Such a function exists.
The property (iii) can be satisfied because, by Remark~\ref{re:(F)}, a family $S_{\delta}^i$ of the combinatorial
type (3) contains a neighborhood of the open edge $\Delta_{1}^2$ in $\Delta^2$ for all small positive $\delta$.
By the curve selection lemma (Lemma 3.2 in \cite{Michel2}) there exists $\delta_0$ and a function $p(t)$ such that
$(p(t_2),t_2)\in S_{\delta_0}^i$ for all $t_2 \in (0,1)$.
Since $S_\delta^i$ is monotone, this implies the inclusion $(\{t_1=p(t_2)\} \cap \Delta^2) \subset S_\delta^i$
for all positive $\delta<\delta_0$.

Finally, there exists a function $=q(t)$ on $[0,1]$ such that
\begin{enumerate}[(i)]
\item
$1>q(t)/(1-t) > 1/2$ for all $t \in (0,1)$ and $q(0)=1/2$;
\item
the point $(t_1,b(\delta))\notin S_\delta^i$ for all small positive $\delta$ and all
$t_1\in[q(b(\delta)),1-b(\delta))$.
\end{enumerate}
Since $S_\delta^i$ is a monotone family and $b(\delta)$ is monotone decreasing for small $\delta$, it follows
that $(t_1,t_2)\notin S_\delta^i$ for all small positive $\delta$ and all
$t_2 \in [b(\delta),1),\> t_1\in[q(t_2),1-t_2)$.

Now we describe the homeomorphism $h^2:\> \overline{\Delta^2} \to \overline{\Delta^2}$.

We set $h^2:=\psi \circ \varphi$, where $\psi$ and $\varphi$ are homeomorphisms
$\overline{\Delta^2} \to \overline{\Delta^2}$ defined as follows.

Let
$$\varphi(t_1,t_2):=((1-h^1(t_2))h^1(t_1/(1-t_2)),h^1(t_2))\> \text{when}\> t_2<1,\>
\text{and}\> \varphi(t_1,1)=(0,1).$$
Then $\varphi$ restricted to all edges of $\Delta^2$
equals to $h^1$ (after identification of these edges with $\Delta^1$).

Next, we define the homeomorphism $\psi$ so that:
\begin{enumerate}
\item
$\psi$ preserves the segments $\{ t_2= {\rm const} \} \cap \Delta^2$,
\item
$\psi(p(t_2),t_2)=((1-t_2)/2,t_2)$ for all $t_2$ close to 1,
\item
$\psi((1-t_2)/2,t_2)=(q(t_2),t_2)$ for all $t_2$ close to 1.
\end{enumerate}
Observe that one can choose $\psi$ to be piecewise linear on each
segment $\{ t_2= {\rm const} \} \cap \Delta^2$.
Then $\psi$ acts as identity on the boundary of $\Delta^2$.

We now prove that the homeomorphism $h^2$ satisfies the property, required in the theorem,
for the case of a family $S_\delta$ combinatorially equivalent to the model family
$V_\delta= \{t_2\le 1-\delta\}\cup\{2 t_1+t_2\le 1\}$ in the equivalence class (3).
The proofs for other combinatorial types are similar (and simpler).

Define
$$A_\delta:=(\{t_2 \le a(\delta)\}\cup\{t_1\le p(t_2)\}) \cap \Delta^2$$
and
$$B_\delta:=(\{t_2 \le b(\delta)\}\cup\{t_1\le q(t_2)\}) \cap \Delta^2.$$
Then $A_\delta$ and $B_\delta$ have the same combinatorial type as $S_\delta$ and $V_\delta$, and
$A_\delta \subset S_\delta \subset B_\delta$ for all small $\delta >0$
(the inclusion $A_\delta \subset S_\delta$ uses the fact that $S_\delta$ is a 2-ball, according to the basic condition (E)).
We now prove that $V_\delta\subset h^2(A_\delta)$ and $B_\delta\subset h^2(V_\delta)$
for all small positive $\delta$, which immediately implies the theorem.

Let
$$A'_\delta := (\{t_2\le 1-\delta\}\cup\{t_1\le p(t_2)\}) \cap \Delta^2$$
and
$$V'_\delta := (\{t_2\le b(\delta)\}\cup\{2 t_1+t_2\le 1\}) \cap \Delta^2.$$
Then it is easy to check, using the properties of functions $p(t_2)$, $q(t_2)$, $a(\delta)$, $b(\delta)$,
that, for all small $\delta >0$, the following inclusions hold.
\begin{enumerate}
\item
$A'_\delta \subset \varphi(A_\delta)$,
\item
$V'_\delta \subset \varphi(V_\delta)$,
\item
$V_\delta\subset\psi(A'_\delta)$,
\item
$B_\delta\subset\psi(V'_\delta)$.
\end{enumerate}

For example, let us prove (1).
According to the construction of the homeomorphism $h^1$, we have $h^1(a(\delta))>1-\delta$, hence
$$(\{ t_2 \le 1- \delta \} \cap \Delta^2) \subset  \varphi (\{ t_2 \le a(\delta) \} \cap \Delta^2).$$
Furthermore, $\varphi$ maps each interval $\{ t_2= {\rm const} \} \cap \Delta^2$ parallel to itself
along the rays through the vertex 2.
Condition that $p(t_2)/(1-t_2)$ is decreasing as a function of $t_2$, guarantees that the image of
the curve $(p(t_2),t_2)$ under $\varphi$ is to the right (along the coordinate $t_1$) of that curve, hence
$$( \{ t_1 \le p(t_2) \} \cap \Delta^2) \subset \varphi (\{ t_1 \le p(t_2) \} \cap \Delta^2),$$
and (1) is proved.
Proofs of the inclusions (2)--(4) are analogous.

Inclusions (1) and (3) imply $V_\delta\subset h^2(A_\delta)$, while (2) and (4) imply
$B_\delta\subset h^2(V_\delta)$.
\end{proof}

\begin{remark}
\begin{enumerate}[(i)]
\item
The statement analogous to Lemma~~\ref{le:equivalencies} or Theorem~\ref{th:equivalencies} for families in
$\Delta^0$ is trivial (the homeomorphism is the identity).
\item
The proof of Theorem~\ref{th:equivalencies} remains valid for the case when the families $S^i_\delta$
satisfying basic conditions, are not necessarily standard, but we do not  need this fact in the sequel.
For $n=2$ there is only one {\em non-standard model family}, of type (5).
The only addition one has to make to accommodate type (5), is in the description of the function $p(t)$.
More precisely, in item (iii) of the description of this function and in the proof of its existence,
the phrase ``combinatorial type (3)'' should be replaced by ``combinatorial type either (3) or (5)''.

Accordingly, Corollary~\ref{cor:equivalencies} and Lemma~\ref{le:strong} below remain true for
not necessarily standard families satisfying basic conditions.
\item
Theorem~\ref{th:equivalencies} can be generalized to standard monotone families in the standard ordered
$m$-simplex $\Delta^m$ for any $m$.
The proof will appear elsewhere.
\end{enumerate}
\end{remark}

\begin{corollary}\label{cor:equivalencies}
Let $m \in \{0,1,2 \}$.
Given a finite set of standard monotone families $\{ S^1_\delta, \ldots ,S^k_\delta \}$ in {\em definable}
$m$-simplices $(\Sigma_i,\> \Phi_{i})$ with $\Phi_{i}:\> \Delta \to \Sigma_i$, there is
a homeomorphism $h:\> \overline \Delta \to \overline \Delta$ such that the homeomorphisms
$$h_{i,j}:\> \Phi_j \circ h \circ \Phi_i^{-1}:\> \overline \Sigma_i \to \overline \Sigma_j\quad
\text{and}\quad
h_{j,i}:\> \Phi_i \circ h \circ \Phi_j^{-1}:\> \overline \Sigma_j \to \overline \Sigma_i$$
provide topological equivalence in pairs of combinatorially equivalent families in the set.

In particular, if two families $S^i_\delta,\> S^j_\delta$ are combinatorially equivalent, then
they are topologically equivalent.
\end{corollary}

\begin{proof}
Apply Theorem~\ref{th:equivalencies} to the set of standard monotone families
$\Phi_{i}^{-1} (S^i_\delta)$, $i=1, \ldots ,k$ in $\Delta$.
\end{proof}

The following lemma shows, in particular, that for $m \le 2$, the property to be
{\em topologically equivalent}
becomes unnecessary in the definition of the {\em strong topological equivalence}
(this is not true in the case $m=3$).

\begin{lemma}\label{le:strong}
For two standard monotone families $S_\delta$ and $V_\delta$ in definable ordered simplices
$\Sigma$ and $\Lambda$ respectively with $\dim \Sigma= \dim \Lambda \le 2$,
the following three statements are equivalent.
\begin{enumerate}
\item
The families are combinatorially equivalent.
\item
The families are topologically equivalent.
\item
The families are strongly topologically equivalent.
\item
For small $\delta>0$ there is a face-preserving homeomorphism $h_\delta:\> \overline \Sigma \to \overline \Lambda$
mapping $S_\delta$ onto $V_\delta$.
\end{enumerate}
\end{lemma}

\begin{proof}
The statement (1) implies (2) by Corollary~\ref{cor:equivalencies}.
The statement (2) implies (1) since the homeomorphisms $h,\> \Phi$ and $\Psi$ are face-preserving.

Let $v_0,\ v_1,\ v_2$  be the vertices of $\Sigma$  having labels $0,\ 1,\ 2$ respectively,
let $w_0,\ w_1,\ w_2$  be the vertices of $\Lambda$  having labels $0,\ 1,\ 2$ respectively.

Due to the basic condition (E), the boundaries $\partial_\Sigma S_\delta$ of $S_\delta$ and
$\partial_\Lambda V_\delta$ of $V_\delta$ in $\Sigma$ and $\Lambda$ respectively,
are homeomorphic to intervals, hence (a very special case of the Sch\"onflies
theorem, Proposition~6.10 in \cite{BGV_JEMS11}) the pairs $( \overline \Sigma,\>
\overline{\partial_\Sigma S_\delta})$ and $(\overline \Lambda,\> \overline{\partial_\Lambda V_\delta})$
are homeomorphic, in particular there is a homeomorphism
$h'_\delta:\> \overline \Sigma \to \overline \Lambda$, mapping $S_\delta$ onto $V_\delta$.
We prove that there is another homeomorphism, $h''_\delta:\> \overline \Lambda \to \overline \Lambda$,
which preserves $V_\delta$ and maps the images $h'_\delta(v_0), h'_\delta(v_1),h'_\delta(v_2)$ in $\Lambda$
of vertices of $\Sigma$ into vertices $w_1,\ w_2,\ w_3$ respectively of $\Lambda$.

Indeed, consider the closed 2-ball $\overline{V_\delta}$, and the boundary $\partial_\Lambda V_\delta$
of $V_\delta$ in $\Lambda$.
The set $\partial_\Lambda V_\delta$ lies in the boundary of $\overline{V_\delta}$ which is a 1-sphere.
It is easy to construct a homeomorphism
$g$, of the boundary of $\overline{V_\delta}$ onto itself such that
the restriction $g|_{\partial_\Lambda V_\delta}$ is the identity,
and $g(h'_\delta(v_i))=w_i$ for each $h'_\delta(v_i) \in \overline{V_\delta}$.
The homeomorphism $g$ can be extended from the boundary to the homeomorphism
$g_1:\> \overline{V_\delta} \to \overline{V_\delta}$ of the whole 2-ball (\cite{RS}, Lemma~1.10).
Repeating the same argument for the closed 2-ball $\overline{\Lambda \setminus V_\delta}$ we obtain a
homeomorphism $g_2$ of this ball onto itself, such that  $g_2|_{\partial_\Lambda V_\delta}$ is the identity
and $g_2(h'_\delta(v_i))=w_i$ for each $h'_\delta(v_i) \in \overline{\Lambda \setminus V_\delta}$.
Define $h''_\delta$ to coincide with $g_1$ on $\overline{V_\delta}$ and with $g_2$ on
$\overline{\Lambda \setminus V_\delta}$.

Taking $h_\delta:= h''_\delta \circ h'_\delta$, we conclude that the statement (1) implies (4).

By Remark~\ref{re:strongly}, the statement (4) implies (1).
Since the statement (3), according to the definition, is the conjunction of (2) and (4), all the
statements (1)--(4) are pair-wise equivalent.
\end{proof}

\begin{corollary}
A monotone family $S_\delta$ in a definable ordered simplex
$\Sigma$ with $\dim \Sigma \le 2$, satisfying the basic conditions (A)--(E) is standard if and only if it is
topologically equivalent to a standard family.
\end{corollary}

\begin{proof}
Follows directly from Definition~\ref{def:standard_family} and Lemma~\ref{le:strong}.
\end{proof}

\section{Constructing families of monotone curves}
\label{sec:monotone-curves}

Let $a_1 < b_1, \ldots ,a_{n}< b_{n}$, where $n \ge 2$, be real numbers and consider an open box
$$B:= \{ (x_1, \ldots ,x_n) \in \Real^{n}|\>  a_i <x_i <b_i \}.$$
Let $X \subset B$ be a monotone cell, with $\dim X=2$,
such that points $${\bf a}:=(a_1, \ldots ,a_{n}),\> {\bf b}:=(b_1, \ldots ,b_{n})$$
belong to $\overline X$.

\begin{theorem}\label{th:curves}
There is a definable family $\{ \gamma_t|\> 0<t<1 \}$ of disjoint open curve intervals (``curves'') in $X$
such that the following properties are satisfied.
\begin{enumerate}[(i)]
\item
The curve $\gamma_t$ for each $t$ is a monotone cell,
that is, each coordinate function among $x_1, \ldots ,x_n$ is monotone on $\gamma_t$.
\item
The endpoints of each $\gamma_t$ are ${\bf a}$ and ${\bf b}$.
\item
The union $\bigcup_{0<t<1} \gamma_t$ coincides with $X$.
\end{enumerate}
\end{theorem}

\begin{proof}
Since at $\bf a$ the minima of all coordinate functions $x_1, \ldots ,x_n$ on $X$ are attained,
on any germ of a smooth curve in $X$ at $\bf a$, all coordinate
functions are monotone increasing from $\bf a$.
Choose generically such a germ $\alpha$ and a point ${\bf c} \in \alpha$.
Consider all level curves through $\bf c$ of all coordinate functions
(the same level curve can correspond to different functions).
Observe that the cyclic order (of coordinate functions) in which the level curves pass through $\bf c$
is the same as the cyclic order in which level curves pass through any other point in $X$, because $X$
is a monotone cell.
The level curves through $\bf c$ divide $X$ into open sectors, and there is a unique pair
of neighbours in the ordered set of level curves such that $\alpha$ belongs to
the intersection of two of sectors bounded by these curves.
Let this pair of neighbours correspond to coordinate functions $x_\ell, x_m$ (if a level curve
corresponds to several coordinate functions, choose any corresponding function).

Let $X':= \rho_{{\rm span} \{x_\ell, x_m \}} (X)$, and
$$B':= \{ (x_\ell,x_m)|\> a_\ell<x_\ell<b_\ell,\> a_m<x_m<b_m \}= \rho_{{\rm span} \{x_\ell, x_m \}} (B).$$
Note that $X'$ is a semi-monotone set and $X' \subset B'$.

It is clear, from the choice of coordinate functions $x_\ell$ and $x_m$, that if on a curve
in $X'$ both  $x_\ell$ and $x_m$ are monotone, then all other coordinate functions are also monotone.
It follows that it is sufficient to construct a family $\{ \gamma'_t|\> 0<t<1 \}$ of curves in $X'$
with properties analogous to (i)--(iii), i.e.,
\begin{itemize}
\item
each $\gamma'_t$ is the graph of a monotone function in either $x_\ell$ or $x_m$;
\item
the endpoints of each $\gamma'_t$ are $(a_\ell,a_m)$ and $(b_\ell,b_m)$;
\item
the union $\bigcup_{0<t<1} \gamma'_t$ coincides with $X'$.
\end{itemize}
Then, for each $t$, define $\gamma_t$ as $(\rho_{{\rm span} \{x_\ell, x_m \}}|_{X})^{-1}(\gamma'_t)$.

Note first that the top and the bottom of the semi-monotone set $X'$ are graphs of non-decreasing functions
on the interval $X'' \subset {\rm span} \{x_\ell \}$.
It follows that for each $c \in \Real$ such that $a_\ell+a_m<c<b_\ell+b_m$, the
line $\{ x_\ell+x_m=c \}$ intersects $X'$ by an open interval
$(p(c)=(p_\ell(c),p_m(c)),\> q(c)=(q_\ell(c),q_m(c)))$ with
$p_\ell(c), p_m(c), q_\ell(c), q_m(c)$ being continuous non-decreasing functions of $c$.
The limits of both $p(c)$ and $q(c)$ as $c \to a_\ell+a_m$
(respectively, as $c \to b_\ell+b_m$) are equal to
$(a_\ell,a_m)$ (respectively, to $(b_\ell,b_m)$).

For each $t \in(0,1)$ define the curve
$$\beta_t:=\{ (1-t) p(c)+t q(c)|\> c \in (a_\ell+a_m, b_\ell+b_m)\} \subset X'.$$
For each $t$ the curve $\beta_t$ is continuous, and connects points $\bf a$ and $\bf b$.
As the parameter $c$ varies from $a_\ell+a_m$ to $b_\ell+b_m$
both coordinates of points on $\beta_t$ are non-decreasing as the functions of $c$,
and strictly increasing everywhere except the parallelograms
in $X'$ where either both $p_\ell(c)$ and $q_\ell(c)$ are constants or both $p_m(c)$ and $q_m(c)$
are constants.
It remains to modify the family of curves inside each such parallelogram $P$.

Suppose that $p_m(c)$ and $q_m(c)$ are constant in $P$, so the curves $\beta_t \cap P$ are horizontal.
The left and the right sides of $P$ can be parameterized by $t \in (0,1)$.
Consider a monotone increasing continuous function $\phi(t):\> [0,1] \to [0,1]$ such that
$\phi (0)=0$, $\phi (1)=1$, and $t < \phi (t)$ for every $t \in (0,1)$ (e.g., strictly concave).

For each $t$, connect a point corresponding to $t$ on the left side of
$P$ with the point corresponding to $\phi(t)$ on the right side.
This gives a family of monotone (with respect to $c$) curves $\alpha_{t,P}$ inside $P$.
The family $\gamma_t$ can now be defined as $\beta_s$ outside all parallelograms and $\alpha_{s,P}$
inside each parallelogram $P$, with the natural re-parameterizations.
\end{proof}

\section{Combinatorial equivalence to standard families}
\label{sec:combinatorial-equivalence}

Let $K \subset \Real^n$ be a compact definable set with $\dim K \le 2$, and
$\{ S_\delta \}_{\delta >0}$ a monotone definable family of compact subsets of $K$.
According to Lemma~\ref{le:S_by_function}, there is a non-negative upper semi-continuous definable
function $f:\> K \to \Real$ such that $S_\delta=\{\x \in K|\> f(\x) \ge \delta\}$.
Moreover, by Remark~\ref{re:S_by_monot}, there is a cylindrical decomposition ${\mathcal D}'$ of
$\Real^{n+1}$, satisfying the frontier condition, and monotone with respect to the function $f$
(Definition~\ref{def:monot_respect_function}).

Let $\mathcal D$ be the cylindrical cell decomposition induced by ${\mathcal D}'$ on $\Real^n$.

\begin{theorem}\label{th:triang}
There is a definable ordered triangulation of $K$, which is a refinement of the decomposition $\mathcal D$,
with the following properties.
\begin{enumerate}[(i)]
\item
Every definable simplex $\Lambda$ of the triangulation is
a monotone cell.
\item
For every simplex $\Lambda$, the restriction $f|_\Lambda$ is a monotone function.
\item
For every two-dimensional simplex $\Lambda$ and every $i=0,1,2$ the function $(f|_\Lambda)_i$
(see Definition~\ref{def:ext}) is monotone.
\item
For every simplex $\Lambda$ the family $S_\delta \cap \Lambda$ satisfies the basic conditions (A)--(E),
and these conditions are hereditary.
\item
For every simplex $\Lambda$
the family $S_\delta \cap \Lambda$ is a standard family (see Definition~\ref{def:standard_family}).
\end{enumerate}
\end{theorem}

The theorem will follow from a series of lemmas.
First we introduce some necessary notations.

Let $Z$ be a cylindrical two-dimensional section cell of ${\mathcal D}'$, and $Y$ corresponding two-dimensional
cell of ${\mathcal D}$, i.e., $Z$ is the graph of $f|_Y$.
Let $\ell$ be minimal positive integer for which $\dim \rho_{\Real^\ell} Y=1$, and $m$, with
$\ell <m  \le n$, be minimal for which $\dim \rho_{\Real^m} Y=2$.
Then, by Lemma~\ref{le:proj_cyl},
$$X:= \rho_{{\rm span} \{ x_\ell,x_m \}} (Y)=\rho_{{\rm span} \{ x_\ell,x_m \}} (Z)$$
is a two-dimensional cylindrical cell, $Y$ and $Z$ are graphs of continuous definable maps on $X$.
Note that $X$ is a semi-monotone set, while $Y$ and $Z$ are graphs of monotone maps on $X$.
Also observe that $\rho_{{\rm span} \{ x_\ell,x_m \}}$ maps $\overline Y$ and $\overline Z$
onto $\overline X$, but $\overline Y$ and $\overline Z$ are not necessarily graphs of some
maps over $\overline X$.
Let $(a,b) \subset {\rm span}\> \{ x_\ell \}$ be
the common image $\rho_{{\rm span}\> \{ x_\ell \}}(Z)=\rho_{{\rm span}\> \{ x_\ell \}}(Y)=
\rho_{{\rm span}\> \{ x_\ell \}}(X)$.

Notice that the bottom $U$ and the top $V$ of $Z$ (see Definition~\ref{def:cyl_cell}) are open curve
intervals in $\Real^{n+1}$, and that
$$\rho_{{\rm span}\> \{ x_\ell \}}(Z)=\rho_{{\rm span}\> \{ x_\ell \}}(Y)=
\rho_{{\rm span}\> \{ x_\ell \}}(X)=$$
$$=\rho_{{\rm span}\> \{ x_\ell \}}(U)=\rho_{{\rm span}\> \{ x_\ell \}}(V)=(a,b).$$
The side wall $W$ of $Z$ is the pre-image $(\rho_{{\rm span}\> \{ x_\ell \}}|_{\overline Z})^{-1} (\{ a, b \})$,
and by Lemma~\ref{le:side_wall} has exactly two connected components each of which is either a single
point or a closed curve interval.

For a 1-dimensional cylindrical cell $C$ of $\mathcal D$, contained in the closure of $Y$,
denote by $f_{Y,C}$ the unique extension, by semicontinuity, of $f|_{Y}$ to $C$.

\begin{definition}\label{def:blowup}
Let $\varphi:\> W \to \Real$ be a continuous function on a definable set $W \subset \Real^n$, having the graph
$\Phi \subset \Real^{n+1}$.
A point $\x \in \overline W \setminus W$ is called a {\em blow-up point} of $\varphi$ if
$\rho_{\Real^n}^{-1}(\x) \cap \overline \Phi$ contains an open interval.

Infimum of a blow-up point $\x$ is $\inf \{y|\> \{ (\x,y) \in \rho_{\Real^n}^{-1}(\x) \cap \overline \Phi \}\}$.
\end{definition}

\begin{lemma}\label{le:hausdorff}
\begin{enumerate}[(i)]
\item
The set $\partial S_\delta \cap Y$, for small $\delta >0$ is either empty or a 1-dimensional regular cell,
i.e., an open curve interval with distinct endpoints on $\partial Y$.
\item
The Hausdorff limit of a non-empty $\partial S_\delta \cap Y$ is either a vertex of one of the 1-dimensional
cylindrical cells on $\partial Y$ or a curve interval in $\partial Y$ with the endpoints at two distinct vertices.
\item
For small $\delta >0$, an endpoint of a non-empty  $\partial S_\delta \cap Y$ is a vertex of one of the
1-dimensional cylindrical cells on $\partial Y$ if and only if this endpoint is a blow-up point of $f|_Y$ with
infimum $0$.
\end{enumerate}
\end{lemma}

\begin{proof}
Due to Proposition~\ref{prop:def_monotone_map}, the intersection $\{ x_{n+1}=\delta \} \cap Z$, for each $\delta$,
is either empty or a monotone cell.
By Proposition~\ref{prop:proj}, the projection
$\rho_{\Real^n} (\{ x_{n+1}=\delta \} \cap Z)= \{ f|_Y=\delta \} \cap Y$
is also either empty or a monotone cell.
According to Remark~\ref{re:moving}, $\{ f|_Y=\delta \} \cap Y$ coincides with $\partial S_\delta \cap Y$.
Hence (i) is proved.

Let the Hausdorff limit of a non-empty $\partial S_\delta \cap Y$ be either a single point $\x$ which is not a vertex
of any 1-dimensional cylindrical cell on $\partial Y$ and belongs to a 1-dimensional cell $C$ on $\partial Y$,
or an interval with at least one endpoint $\x$ belonging to a 1-dimensional cell $C$ on $\partial Y$.
Then $f_{Y,C}(\x)=0$ while at some other point in the neighbourhood of $\x$ in $C$ the function
$f_{Y,C}$ is positive.
This contradicts the supposition that the cylindrical decomposition ${\mathcal D}'$
satisfies the frontier condition, and is monotone with respect to the function $f$
(see Definition~\ref{def:monot_respect_function}).
This proves (ii).

The item (iii) follows directly from item (ii) and Definition~\ref{def:blowup}.
\end{proof}

\begin{definition}\label{def:separatrix}
Let $Z$ be a two-dimensional cell of the decomposition ${\mathcal D}'$, $i \in \{1, \ldots ,n+1 \}$, and $c \in \Real$.
The intersection $Z \cap \{ x_i=c \}$ is called a {\em separatrix in} $Z$ if
\begin{itemize}
\item
the sets $\{ x_i>c \} \cap Z$ and $\{ x_i<c \} \cap Z$ are not empty;
\item
the set $\{x_i=c\} \cap \overline Z$ is not contained in $\overline{\{x_i > c\} \cap  Z}
\cap \overline{\{x_i<c\} \cap Z}$.
\end{itemize}

If $Z \cap \{ x_i=c \}$ is a separatrix in $Z$ then its {\em extension} is
$\overline Z \cap \{ x_i=c \}$.
\end{definition}

Observe that for each fixed $i=1, \ldots n+1$ any two different separatrices are disjoint and their extensions
are disjoint.

\begin{lemma}\label{le:no_separatrix1}
If a cylindrical cell $Z$ has no separatrices, then in any cylindrical decomposition $\mathcal C$ of
$\Real^{n+1}$ compatible with $Z$, each cylindrical cell contained in $Z$ has no separatrices.
\end{lemma}

\begin{proof}
Suppose, contrary to the claim of the lemma, that there is a cylindrical 2-cell $C$ of $\mathcal C$ which
contains a separatrix $Z \cap \{ x_i=c \}$.

By Definition~\ref{def:separatrix}, there is a point ${\bf u}=(u_1, \ldots ,u_{n+1}) \in \overline C$ with
$u_i=c$ such that the closure of one of the sets, $\{ x_i>c \} \cap C$ or $\{ x_i<c \} \cap C$,
contains ${\bf u}$, while the closure of the other one does not contain ${\bf u}$.

The point $\bf u$ can belong neither to the top $V$ nor to the bottom $U$ of $C$.
Indeed, $V$ is a monotone cell, hence if ${\bf u} \in V$, then the function $x_i$ is constant on $V$.
It follows that either $x_i$ is constant on whole $C$, or $x_i > c$, or $x_i <c$ on $C$.
Any of these alternatives contradicts to $C \cap \{ x_i=c \}$ being a separatrix.
The same argument shows that ${\bf u} \not\in U$.

Now let the point $\bf u$ belong to one of the two connected components $W$ of the side wall of $C$.
Observe that $W=Z \cap \{ x_j=a \}$ for some $j \neq i$ and $a \in \Real$.
The function $x_j$ is a monotone function on the monotone cell $Z \cap \{ x_i=c\}$.
On the other hand, it is constant on some interval in $Z \cap \{ x_i=c\}$ and non-constant on the whole
$Z \cap \{ x_i=c\}$, which is a contradiction.
\end{proof}

\begin{lemma}\label{le:no_separatrix}
There is a refinement ${\mathcal E}'$ of the cell decomposition ${\mathcal D}'$ monotone with respect to $f$,
and such that no 2-cell $Z$ in ${\mathcal E}'$ contains a separatrix.
\end{lemma}

\begin{proof}
First notice that for given $Z$ and $i$ there is a finite number of separatrices.
Indeed, the 1-dimensional set $A$ of points ${\bf u}=(u_1, \ldots ,u_{n+1}) \in \partial Z$ such that
$Z \cap \{ x_i=u_i \}$ is a separatrix, is definable, hence has a finite number of connected components.
Since $A$ and each separatrix extension are 1-dimensional and compact, each connected component of
$A$ is contained in one of these extensions.
It follows that the number of separatrix extensions, and hence separatrices, does not exceed the
number of connected components of $A$.

Apply Corollary~\ref{cor:refinement} to the set of all separatrices of all 2-cells of ${\mathcal D}'$
(as the sets $U_1, \ldots ,U_m$ in the corollary) and the decomposition ${\mathcal D}'$ (as ${\mathcal A}$).
By the corollary, there is a refinement ${\mathcal E}'$ of ${\mathcal D}'$,
monotone with respect to the function $f$ and $U_1, \ldots ,U_m$, such that no 2-cell $Z$ of
${\mathcal E}'$ intersects with a separatrix of any 2-cell of ${\mathcal D}'$.

According to Lemma~\ref{le:no_separatrix1}, no 2-cell in ${\mathcal E}'$ contains a separatrix.
\end{proof}

\begin{lemma}\label{le:edges}
Suppose that a 2-cell $Z$ in a cylindrical decomposition ${\mathcal E}'$ does not contain a separatrix.
Then any point ${\bf v}=(v_1, \ldots ,v_{n+1}) \in Z$ can be connected to any point
${\bf u}=(u_1, \ldots ,u_{n+1}) \in \overline Z$ by a curve $\gamma$ which is 1-dimensional monotone cell.
\end{lemma}

\begin{proof}
The point ${\bf v}$ belongs to a certain sign condition set of the functions $x_i-u_i$
for all $i=1, \ldots ,n+1$.
Let, for definiteness, ${\bf v} \in \{ x_1>u_1, \ldots , x_{n+1}>u_{n+1} \}$.
Let
$$B:= \{ (x_1, \ldots ,x_{n+1}) \in \Real^{n+1}|\>  u_i <x_i <v_i \}.$$
Observe that $Z \cap B$ is a monotone cell.
The closure $\overline{ Z \cap B}$ contains ${\bf u}$, otherwise there would be a separatrix
$Z \cap \{ x_i=u_i \}$ for some $i$, which contradicts the main property of the cylindrical
decomposition ${\mathcal E}'$ (Lemma~\ref{le:no_separatrix}).
Applying Theorem~\ref{th:curves} to $Z \cap B$ as $\bf F$, and $({\bf v},\> {\bf u})$ as
$({\bf a},\> {\bf b})$, connect ${\bf v}$ to ${\bf u}$ by a curve $\gamma$, which
is a monotone cell.
\end{proof}

\begin{lemma}\label{le:triang_up}
Let $Z$ be a cylindrical cell of ${\mathcal E}'$ having side wall $W'$, bottom $U'$, and top $V'$.
Let $P'$ be a finite subset of $W'$, containing all vertices of $Z$, and ${\bf v}'=(v'_1, \ldots ,v'_{n+1})$
a point in $Z$.
Introduce ${\bf a}':=U' \cap \{ x_\ell=v'_\ell \}$ and ${\bf b}':=V' \cap \{ x_\ell=v'_\ell \}$, where
$\ell$ is the minimal positive integer for which $\dim \rho_{\Real^\ell} Z=1$.
There is a definable triangulation ${\mathcal B}'$ of $\overline Z$ such that
\begin{enumerate}[(i)]
\item
the triangulation ${\mathcal B}'$ is a cylindrical decomposition of $\overline Z$, in particular
each simplex is a cylindrical cell;
\item
each simplex is a monotone cell;
\item
each 2-simplex does not contain a separatrix;
\item
the set of all vertices of the triangulation is
$P' \cup \{ {\bf v}',\> {\bf a}',\> {\bf b}' \}$;
\item
the edges connecting ${\bf v}'$ with ${\bf a}'$ and ${\bf b}'$ are contained in $\{ x_\ell=v'_\ell \}$.
\end{enumerate}
\end{lemma}

\begin{proof}
If $\bf u$ is any point in $P'$, then by Lemma~\ref{le:edges}, there exists a curve $\gamma$ which is 1-dimensional
monotone cell, connecting points ${\bf v}'$ and $\bf u$.
If $\bf u$ be either ${\bf a}'$ or ${\bf b}'$, then $u_\ell=v'_\ell$ and we connect ${\bf v}'$ to $\bf u$
by a monotone cell (interval) $\gamma \subset \{ x_\ell=v'_\ell \} \cap Z$.
Thus we have constructed a triangulation of $\overline Z$, denote it by ${\mathcal B}'$.

Each $\Lambda'$ in ${\mathcal B}'$ is a cylindrical cell.
Indeed, since all 1-dimensional simplices are monotone cells, each such simplex, except the ones contained in
$\{ x_\ell=v'_\ell \} \cap Z$, is the graph of a monotone function on an interval in ${\rm span} \{ x_\ell \}$,
hence is a cylindrical 1-cell.
Each two-dimensional simplex has the top and the bottom among these graphs,
moreover the corresponding two functions are defined on the same interval.
Hence each two-dimensional simplex is a cylindrical two-cell.
According to the Lemma~\ref{le:no_separatrix1}, each two-dimensional simplex contains no separatrix.

Now we show that each two-dimensional $\Lambda'$ in ${\mathcal B}'$ is a monotone cell.
Let $\Lambda'$ be the definable simplex contained in the monotone cell $Z \cap \{ x_\ell <v'_\ell \}$,
bounded in $Z$ by $Z \cap \{x_\ell=v'_\ell \}$ and the curve $\gamma$, connecting ${\bf v}'$
with a vertex of $Z$.
Then, by Theorem~11 in \cite{BGV2}, the curve $\gamma$ divides $Z \cap \{ x_\ell <v'_\ell \}$ into two
monotone cells, one of which is $\Lambda'$, and the other is the union of the remaining
definable two-dimensional simplices and their boundaries in $Z \cap \{ x_\ell <v'_\ell \}$.
Applying inductively the same argument, we prove that each of these remaining simplices is a monotone cell,
and all simplices in $Z \cap \{ x_\ell > v'_\ell \}$ are also monotone cells.
Hence each simplex $\Lambda'$ is a monotone cell.
\end{proof}

\begin{corollary}\label{cor:triang_down}
Let $Y$ be a cylindrical cell of ${\mathcal E}$ having side wall $W$, bottom $U$, and top $V$.
Let $P$ be a finite subset of $W$, containing all vertices of $Y$, and ${\bf v}=(v_1, \ldots ,v_n)$ a point in $Y$.
Let $Q \subset P$ be the subset of all points ${\bf w} \in P$ at which the function $f|_Y$ has a blow-up,
and for each ${\bf w} \in Q$ fix one of the limit values $\alpha$ of $f|_Y$ at $\bf w$.
Introduce ${\bf a}:=U \cap \{ x_\ell=v_\ell \}$ and ${\bf b}:=V \cap \{ x_\ell=v_\ell \}$, where
$\ell$ is the minimal positive integer for which $\dim \rho_{\Real^\ell} Y=1$.
There is a definable triangulation $\mathcal B$ of $\overline Y$ such that
\begin{enumerate}[(i)]
\item
the triangulation ${\mathcal B}$ is a cylindrical decomposition of $Y$, in particular
each simplex is a cylindrical cell;
\item
each simplex is a monotone cell;
\item
each 2-simplex does not contain a separatrix;
\item
the set of all vertices of the triangulation is
$P \cup \{ {\bf v},\> {\bf a},\> {\bf b} \}$;
\item
the edges connecting $\bf v$ with ${\bf a}$ and ${\bf b}$ are contained in $\{ x_\ell=v_\ell \}$;
\item
for every simplex $\Lambda$, the restriction $f|_\Lambda$ is a monotone function;
\item
for each ${\bf w} \in Q$, the limit of $f|_Y$ at $\bf w$ along the edge connecting $\bf v$ and
$\bf w$ equals $\alpha$.
\end{enumerate}
\end{corollary}

\begin{proof}
Let $Z$ be a cylindrical cell of ${\mathcal E}'$ such that $\rho_{\Real^n} (Z)=Y$, i.e., $Z$ is the graph
of $f|_Y$.
Since $Q$ is the set of all blow-up points of $f|_Y$, for each point
${\bf u} \in P \setminus Q$ there is a unique pre-image $\rho^{-1}({\bf u})$ in the side wall of $Z$.
Applying Lemma~\ref{le:triang_up} to $Z$ and the union of the set
$(\rho_{\Real^n}|_{\overline Z})^{-1}(P \setminus {\bf w}) \cup \{ ({\bf w}, \alpha)|\> {\bf w} \in Q \}$
and the set of all vertices of $Z$ as $P'$, and the point $(\rho_{\Real^n}|_{Z})^{-1}({\bf v})$ as ${\bf v}'$,
we obtain a triangulation ${\mathcal B}'$ of $Z$.
In particular, all 1-dimensional simplices of ${\mathcal B}'$ are monotone cells.
Projections by $\rho_{\Real^n}$ of these 1-dimensional simplices are monotone cells, connecting $\bf v$
with the points in $P$, ${\bf a}$ and ${\bf b}$.

The properties (i)--(v) of the triangulation $\mathcal B$ can be proved exactly as the analogous properties
of the triangulation ${\mathcal B}'$ in Lemma~\ref{le:triang_up}.

For each simplex $\Lambda$ in $\mathcal B$, the function $f|_\Lambda$ is monotone since its graph is
a simplex $\Lambda'$  of the triangulation ${\mathcal B}'$, and hence is a monotone cell.

Finally, the property (vii) of $\mathcal B$ is valid by the choice of edges connecting $\bf v$ and
points ${\bf w} \in Q$.
\end{proof}

The triangulation constructed in Corollary~\ref{cor:triang_down} is not ordered.
To label the vertices of simplices so that conditions (iv) and (v) of Theorem~\ref{th:triang} are satisfied we
will need to perform a further refinement of the triangulation $\mathcal B$, as follows.

Let $Y$ be a cylindrical cell of the cell decomposition $\mathcal E$.
For the side wall $W$ of $Y$ define the finite set $P \subset W$ as the set of all vertices of cylindrical cells
of $\mathcal E$, contained in $W$ (including vertices of $Y$).

Apply Corollary~\ref{cor:triang_down} to $Y$, with this $P$, and arbitrary $\bf v$.
If there is a vertex $\bf w$ of $Y$ where $f|_Y$ has a blow-up point with infimum $0$, then choose the value
$\alpha =0$.
For all other points ${\bf w} \in Q$ choose $\alpha$ arbitrarily.
According to the corollary, there is a definable triangulation ${\mathcal B}(Y)$ of $\overline Y$ satisfying
properties (i)--(vii).

Let $\Lambda$ be a simplex of the triangulation ${\mathcal B}(Y)$.
By Corollary~\ref{cor:triang_down}, (i), the simplex $\Lambda$ is a cylindrical cell.
One of the connected components of its side wall is a curve interval, while another is a single point.
Apply Corollary~\ref{cor:triang_down} to $\Lambda$ (as $Y$), with the set $P$ consisting of all three vertices
of $\Lambda$ and the point in the middle of the 1-dimensional component of its side wall.
If any of the vertices of $\Lambda$ is a blow-up point $\bf w$ of $f|_\Lambda$ with infimum $0$, then choose a value
$\alpha >0$.
The resulting definable triangulation ${\mathcal B}(\Lambda)$ is then a barycentric subdivision of $\Lambda$.
Label the vertices of ${\mathcal B}(\Lambda)$ as follows: the center of the subdivision is
assigned 0, the vertices of $\Lambda$ are assigned 2, while the remaining three vertices are assigned 1.

Observe that any simplex $\Sigma$ of the triangulation ${\mathcal B}(\Lambda)$ may have at most one vertex at
which $f|_Y$ has a blow-up, and this vertex is $\Sigma_{0,1}$ (having the label 2).

This construction is illustrated on Figire~\ref{fig:gen-barycentric-0}.
In this picture the triangulation ${\mathcal B}(\Lambda)$ is shown for just one simplex $\Lambda$, the only simplex
of the triangulation ${\mathcal B}(Y)$ to which the restriction of the family $S_\delta$ is not separable.

\begin{figure}[hbt]
\vspace*{0.2in}
       \centerline{
          \scalebox{0.45}{
             \includegraphics{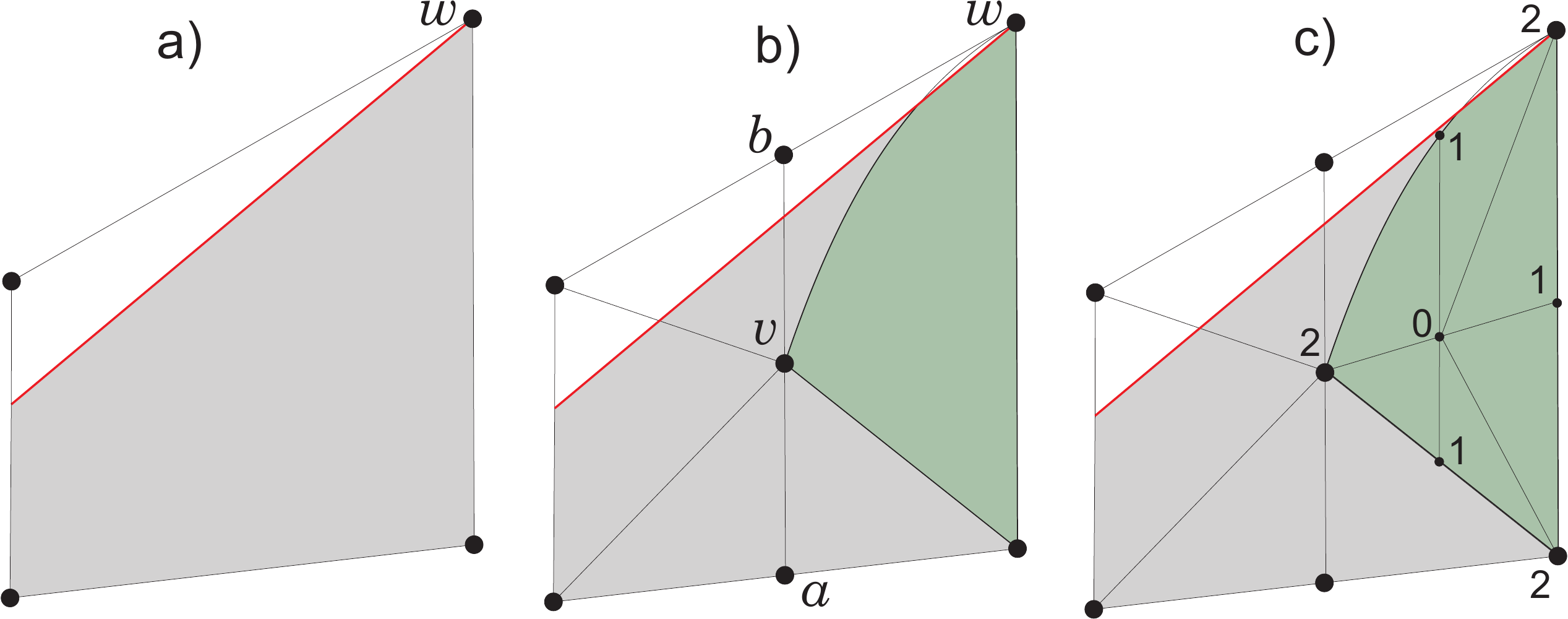}
             }
           }
\vspace*{-0.1in}
\caption{Iterated stellar subdivision of $Y$.}
\label{fig:gen-barycentric-0}
\end{figure}

According to Convention~\ref{con:family_in_simplex}, for a simplex $\Sigma$ of
${\mathcal B}(\Lambda)$ we will write $S_\delta$ meaning the monotone family $\{ S_\delta \cap \Sigma \}$.

\begin{lemma}\label{le:edges1}
For all cylindrical cells $Y$ in $\mathcal E$, for all simplices $\Lambda$ in ${\mathcal B}(Y)$,
for all 1-dimensional simplices $\Sigma$ in ${\mathcal B}(\Lambda)$, and for all $i=0,1,2$
the restriction $(S_\delta)_i$ has one of the three combinatorial types of 1-dimensional model families.
\end{lemma}

\begin{proof}
Since the cylindrical decomposition ${\mathcal D}'$ is monotone with respect to
the function $f$ (see Definition~\ref{def:monot_respect_function}) the vertex $\Sigma_{1,2}$ (labelled by 0)
belongs to
$S_\delta \cap Y$ for small $\delta >0$ whenever $S_\delta \cap Y \neq \emptyset$.
It follows that the restrictions $(S_\delta)_1$ and $(S_\delta)_2$ have the combinatorial type either
(1) or (2) of 1-dimensional model family.
If in the edge $\Sigma_0$ the vertex $\Sigma_{0,2}$ (labelled by 1) lies in $Y$, then $(S_\delta)_0$
has the type either (1) or (2) by the same argument.
If, on the other hand, in $\Sigma_0$ the vertex $\Sigma_{0,2}$ lies in $\partial Y$, then it, and the whole
edge $\Sigma_0$ lies in a 1-dimensional cell $C$ of the decomposition $\mathcal E$, and vertices of $C$
are labelled by 2, in particular the vertex $\Sigma_{0,2}$ lies in $C$.
By Lemma~\ref{le:hausdorff}, (iii), $C$ is compatible with the Hausdorff limit of $\partial_{\Sigma}S_\delta$,
hence $(S_\delta)_0$ may have any of the three combinatorial types of 1-dimensional model family.
\end{proof}

\begin{lemma}\label{le:basic_sigma}
For all cylindrical cells $Y$ in $\mathcal E$, for all simplices $\Lambda$ in ${\mathcal B}(Y)$,
and for all simplices $\Sigma$ in ${\mathcal B}(\Lambda)$, the family $S_\delta$
satisfies the basic conditions (A)--(E), and these conditions are hereditary.
\end{lemma}

\begin{proof}
Recall that by Definition~\ref{def:standard_family}, the standard family satisfies the basic conditions (A)--(E).

Since the cylindrical decomposition ${\mathcal D}'$ is monotone with respect to
the function $f$ (see Definition~\ref{def:monot_respect_function}) the vertex $\Sigma_{1,2}$ (labelled by 0)
belongs to
$S_\delta \cap Y$ for small $\delta >0$ whenever $S_\delta \cap Y \neq \emptyset$.
It follows that the basic condition (A) is satisfied.

Since, by the construction, $\Sigma$ may have at most one vertex at which $f|_Y$ has a blow-up, and this
vertex is $\Sigma_{0,1}$ (having the label 2), the vertex $\Sigma_{0,2}$ (having the label 1) is not
a blow-up vertex.
Thus, it cannot happen that one of the two sets
$(S_\delta)_{0,2},\> (S_\delta)_{2,0}$ is empty while another is not.
It follows that the basic condition (C) is satisfied.

Because of the basic condition (C), the only possibility for the basic condition (B) to fail
for a simplex $\Sigma$ would be to have $(S_\delta)_0= \Sigma_0$ while $(S_\delta)_{1,0} = \emptyset$
for small $\delta >0$.
In this case the vertex $\Sigma_{0,1}$ of $\Sigma$, labeled by 2, is a blow-up point of $f|_\Sigma$.
By the construction, the function $f|_{\Sigma_1}(\x) \to \alpha >0$ as $\x \to \Sigma_{0,1}$.
Since $f_Y$ is monotone, hence continuous, we have $f_{\Sigma,\Sigma_1}=f|_{\Sigma_1}$.
It follows that $(S_\delta)_{1,0} \neq \emptyset$, for small $\delta >0$, which is a contradiction.

The basic condition (D) holds true since the cylindrical decomposition ${\mathcal D}'$ is monotone with respect to
the function $f$.

Let $\Sigma'$ be the graph of the function $f|_\Sigma$.
Due to Proposition~\ref{prop:def_monotone_map}, for every $\Sigma$, the intersections
$\Sigma' \cap \{ x_{n+1}=c \}$,
$\Sigma' \cap \{ x_{n+1} \lessgtr c \}$ for each $c \in \Real$ are either empty or monotone cells.
By Proposition~\ref{prop:proj}, the projections of these sets to $\Real^n$, in particular the $\delta$-level sets
$\partial_{\Sigma}S_\delta$ of the function $f|_{\Sigma}$, are also either empty or monotone cells.
Hence, the basic condition (E) is satisfied.

All basic conditions are hereditary for the restrictions of $S_\delta$ to edges of $\Sigma$ by Lemma~\ref{le:edges1}.
\end{proof}

\begin{lemma}\label{le:standard_sigma}
For all cylindrical cells $Y$ in $\mathcal E$, for all simplices $\Lambda$ in ${\mathcal B}(Y)$,
and for all simplices $\Sigma$ in ${\mathcal B}(\Lambda)$, the family $S_\delta$ is standard.
\end{lemma}

\begin{proof}
By Lemma~\ref{le:boolean_monotone}, since the basic conditions (A)--(E) are satisfied for the family
$S_\delta$ for every $\Sigma$, it belongs to either one of the standard combinatorial types, or to the combinatorial
type (5) of the non-standard model family
$$(\{t_1+t_2\le 1-\delta\}\cup\{2 t_1+t_2\le 1\})\cap\Delta^2.$$
However the latter alternative is not possible.
Suppose this is the case.
Then the vertex $\Sigma_{0,1}$ in $\Sigma$ (labeled by 2), is a blow-up point $\bf w$ of $f|_Y$.
Since $\Sigma$ is an element of a barycentric subdivision of some $\Lambda$, it is one simplex of the two, in the
subdivision, having the vertex $\bf w$.

First let $\Sigma$ be the simplex whose edge $\Sigma_0$ lies in the edge of $\Lambda$ connecting  the internal
point of $Y$ with $\bf w$.
Then the vertex $\Sigma_{0,2}$ (labeled by 1) of $\Sigma$ belongs to $S_\delta$ for small $\delta >0$, which
contradicts to the family $S_\delta$ being of the type (5).

Now let $\Sigma$ be the other simplex in the barycentric subdivision of $\Lambda$ with the vertex $\bf w$.
Recall that, by the construction, $f|_Y(\x) \to 0$ as $\x \to {\bf w}$ along the edge $\gamma$ of $\Lambda$ connecting
the internal point of $Y$ with $\bf w$.
Under the supposition that $S_\delta$ is of the type (5), at each point $\x \in \gamma$ in the neighbourhood of $\bf w$
we have $f|_Y(\x)> \delta$ for small $\delta >0$ which is a contradiction.
\end{proof}

\begin{proof}[Proof of Theorem~\ref{th:triang}]
The theorem  follows immediately from Lemma~\ref{le:no_separatrix}, Corollary~\ref{cor:triang_down},
Lemmas~\ref{le:basic_sigma} and \ref{le:standard_sigma}.
\end{proof}

Corollary~\ref{cor:equivalencies} and Theorem~\ref{th:triang} immediately imply the following theorem.

\begin{theorem}\label{th:main}
When $\dim K \le 2$ there exists a definable triangulation of $K$ such that
for each simplex $\Lambda$ and small $\delta >0$, the intersection
$S_\delta \cap \Lambda$ is topologically equivalent to
one of the model families $V_\delta$ in the standard simplex $\Delta$.
\end{theorem}

\section{Triangulations with separable families}\label{sec:triang_sep}

Consider the triangulation of $K$ from Theorem~\ref{th:main}, i.e., the definable homeomorphism
$\Phi:\> |C| \to K$, where $C$ is the finite ordered simplicial complex.
Let $\Sigma_1, \ldots ,\Sigma_r$ be all 1-simplices in $C$ and $\Delta_1, \ldots ,\Delta_k$ be all
2-simplices in $C$.
Then there are affine face-preserving homeomorphisms $\Psi_i:\> \Delta^1 \to \Sigma_i$ and
$\Phi_i:\> \Delta^2 \to \Delta_i$, where $\Delta^1$ and $\Delta^2$ are a standard ordered 1-simplex and
2-simplex respectively.
Let $\Lambda^1_i:=\Phi(\Sigma_i)$, $\Lambda^2_j:=\Phi(\Delta_j)$.
In $\Delta^1$ consider definable families $T^i_\delta:=\Psi_i^{-1}(\Phi^{-1}(S_\delta \cap \Lambda^1_i))$, and
in $\Delta^2$ consider definable families $S^i_\delta:=\Phi_i^{-1}(\Phi^{-1}(S_\delta \cap \Lambda^2_i))$.
For each $i=1, \ldots, r$, in $\Delta^1$ consider the model family $W^i_\delta$ combinatorially equivalent to
the family $T^i_\delta$, and for each $j=1, \ldots k$ in $\Delta^2$ consider the model family $V^j_\delta$
combinatorially equivalent to the family $S^j_\delta$.
Let $V_\delta:= \bigcup_{i} \Psi_i(W^i_\delta) \cup \bigcup_j \Phi_j(V^j_\delta)$.

\begin{lemma}\label{le:homotopy_equivalent}
For all small positive $\delta$, the sets $\Phi^{-1}(S_\delta)$ and $V_\delta$ are homotopy equivalent.
\end{lemma}

\begin{proof}
Without loss of generality, assume that $S_\delta$ (and hence, $\Phi^{-1}(S_\delta)$ and $V_\delta$) is connected.
Consider the covering of the set $\Phi^{-1}(S_\delta)$ by all sets $A_\delta$ in
$$\{ \overline{\Psi_i(T^i_\delta)},\> \overline{\Phi_i(S^j_\delta)}|\> i=1, \ldots ,r,\> j=1, \ldots ,k \}.$$
and let ${\mathcal N}_{\Phi^{-1}(S_\delta)}$ be the nerve of this covering.
Observe that every finite non-empty intersection $A^{i_1}_\delta \cap \cdots \cap A^{i_t}_\delta$ is either
a single point or an interval in one of 1-simplices $\Sigma_i$, hence contractible.
By the Nerve Theorem (\cite{Bjorner}, Theorem~6), $\Phi^{-1}(S_\delta)$ is homotopy equivalent to
the geometric realization of ${\mathcal N}_{\Phi^{-1}(S_\delta)}$.

Analogously, the covering of the set $V_\delta$ by all sets $B_\delta$ in
$$\{ \overline{\Psi_i(W^i_\delta)},\> \overline{\Phi_i(V^j_\delta)}|\> i=1, \ldots ,r,\> j=1, \ldots ,k \}$$
has the nerve ${\mathcal N}_{V_\delta}$ whose geometric realization is homotopy equivalent to $V_\delta$.

Since each model family $W^i_\delta$ is combinatorially equivalent to $T^i_\delta$, and each model family
$V^j_\delta$ is combinatorially equivalent to $S^j_\delta$, the nerves ${\mathcal N}_{\Phi^{-1}(S_\delta)}$
and ${\mathcal N}_{V_\delta}$ are isomorphic, hence their geometric realizations are homotopy equivalent.
It follows that $\Phi^{-1}(S_\delta)$ and $V_\delta$ are homotopy equivalent.
\end{proof}

\begin{theorem}
There exists a definable monotone family $R_\delta$ in $K$ and a definable ordered triangulation of $K$ such that
$R_\delta$ is homotopy equivalent to $S_\delta$ for all small positive $\delta$, and
for each ordered simplex $\Lambda$ and each small enough $\delta >0$, the intersection $R_\delta \cap \Lambda$
is topologically equivalent to one of the {\em separable} families in the standard simplex.
\end{theorem}

\begin{proof}
By Lemma~\ref{le:homotopy_equivalent}, for all small positive $\delta$, the sets $\Phi^{-1}(S_\delta)$ and
$V_\delta$ are homotopy equivalent.
According to Lemma~\ref{le:model_bary}, the restrictions of $W^i_\delta$ to each simplex of the barycentric
subdivision of $\Delta^1$, and of $V^j_\delta$ to each simplex of the barycentric subdivision
of $\Delta^2$ are separable model families.
As $R_\delta$ take $\Phi(V_\delta)$ and as triangulation of $K$ take $\Phi:\> |C'| \to K$, where $C'$
is the barycentric subdivision of the simplicial complex $C$.
\end{proof}

\section{Approximation by compact families}
\label{sec:approximation}
Let, as before, $S:=\bigcup_{\delta >0}S_\delta \subset K \subset \Real^n$.
In \cite{GV07} the following construction was introduced.

For each $\delta>0$, let $\{ S_{\delta, \eps} \}_{\eps >0}$ be a definable family of compact
subsets of $K$
such that the following conditions hold:
\begin{enumerate}
\item
for all $\eps,\> \eps' \in (0,1)$, if $\eps' > \eps$, then $S_{\delta, \eps} \subset S_{\delta, \eps'}$;
\item
$S_\delta= \bigcap_{\eps >0} S_{\delta, \eps}$;
\item
for all $\delta' >0$ sufficiently smaller than $\delta$, and all $\eps' >0$, there exists an open in $K$ set
$U \subset K$ such that $S_\delta \subset U \subset S_{\delta', \eps'}$.
\end{enumerate}

For a sequence $\eps_0 \ll \delta_0 \ll \eps_1 \ll \delta_1 \ll \cdots \ll  \eps_m \ll \delta_m$,
introduce the compact set $T_m(S_{\delta, \eps}):= S_{\delta_0, \eps_0} \cup \cdots \cup S_{\delta_m, \eps_m}$.
Here $m \ge 0$, and $\ll$ stands for ``sufficiently smaller than'' (for the precise meaning of $\ll$ see
Definition~1.7 in \cite{GV07}).

Let $H_i(X)$ be the singular homology group of a topological space $X$ with coefficients in some fixed
Abelian group.
Without loss of generality, assume that $S$ is connected in order to make the homotopy groups
$\pi_k(S)$ and $\pi_k(T_m(S_{\delta, \eps}))$ independent of a base point.

\begin{proposition}[\cite{GV07}, Theorem~1.10]\label{pr:epi}
\begin{enumerate}[(i)]
\item
For every $1 \le k \le m$, there are epimorphisms
$$\psi_k:\> \pi_k(T_m(S_{\delta, \eps})) \to \pi_k(S),$$
$$\varphi_k:\> H_k(T_m(S_{\delta, \eps})) \to H_k(S).$$
\item
If there is a triangulation of $S$ such that the restriction of $S_\delta$ to each simplex is separable, then
$\psi_k$ and $\varphi_k$ are isomorphisms for all $k \le m-1$.
Herewith if $m \ge \dim S$, then $T_m(S_{\delta, \eps})$ is homotopy equivalent to $S$.
\end{enumerate}
\end{proposition}

It was conjectured in \cite{GV07}, Remark~1.11, that $\psi_k$ and $\varphi_k$ are isomorphisms for
all $k \le m-1$ even without the separability condition.

We now show that this conjecture is true in case when $\dim K \le 2$.

Let a family $\{ S_{\delta, \eps} \}$ have the corresponding monotone family $\{ S_\delta \}$.
Consider the triangulation from Theorem~\ref{th:triang}, corresponding to $\{ S_\delta \}$.
Construct new families $\{ V_\delta \}$ and $\{ V_{\delta, \eps } \}$ in $K$ as follows.

We start with $V_\delta$.
For each definable simplex $\Lambda$ of the triangulation $\Phi$, if the restriction $\Lambda \cap S_\delta$ is
separable, let $\Lambda \cap V_\delta$ coincide with $\Lambda \cap S_\delta$.
If $\Lambda \cap S_\delta$ is not separable, consider two cases.

In the first case $\Lambda_0 \subset S_{\delta, \eps}$.
Then, for $\Lambda \cap V_\delta$, replace the restriction $\Lambda \cap S_\delta$ by the (separable) family of combinatorial
type of the model family (4) (i.e., by the full simplex $\Lambda$).

In the second case $\Lambda_0 \not\subset S_{\delta, \eps}$.
Let $\Phi:\> \overline \Delta \to \overline \Lambda$, where $\Delta$ is the standard 2-simplex, be a homeomorphism
preserving the order of vertices.
Let $W_\delta$ be the monotone family in $\Delta$ having the combinatorial type of the model family (2),
such that $\partial_{\Delta}W_\delta$ is the straight line interval parallel to $\Delta_2$ with the endpoint
$\Phi^{-1}(\partial_{\Lambda_0}(\overline{\Lambda \cap S_\delta} \cap \Lambda_0))$
(hence, the restrictions to $\Delta_0$ of $\Phi^{-1}(\Lambda \cap S_\delta)$ and $W_\delta$ coincide).
For $\Lambda \cap V_\delta$, replace the restriction $\Lambda \cap S_\delta$ by a (separable) family $\Phi (W_\delta)$.

Now we construct the family $V_{\delta, \eps }$.
Observe that the set $\Phi^{-1}(\overline{\Lambda \cap S_{\delta, \eps}} \cap \Lambda_0)$ consists of two semi-open
intervals $(\Delta_{0,2}, a]$ and $[b, \Delta_{0,1})$ in $\Delta_0$
(in particular, the open endpoint of the first interval is the vertex 1 of $\Delta$, while the open endpoint of the
second interval is the vertex 2).
Let $\alpha$ be the intersection of $\Delta$ with the straight line passing through $a$ parallel to $\Delta_2$, and
$\beta$ be the intersection of $\Delta$ with the straight line passing through $b$ parallel to $\Delta_1$.

Let $\Lambda^1, \ldots ,\Lambda^r$ be all two-dimensional simplices of the triangulation such that
$\Lambda^i \cap S_\delta$ is non-separable and $\Lambda^i \cap V_\delta \neq \Lambda^i$ for each $i=1, \ldots ,r$.
For each $\delta$ let $V_{\delta, \eps} \setminus \bigcup_i \Lambda^i$ coincide with
$S_{\delta, \eps} \setminus \bigcup_i \Lambda^i$.

For each $i=1, \ldots ,r$ let $\Lambda^i \cap V_{\delta, \eps}$ be the union of two sectors, closed in $\Lambda^i$.
One sector lies between $\Lambda_2$ and the curve $\Phi(\alpha)$.
Another sector lies between $\Lambda_1$ and the curve $\Phi(\beta)$.

\begin{remark}\label{re:tele}
\begin{enumerate}[(i)]
\item
The family $V_{\delta, \eps}$ satisfies the conditions (1), (2), (3) at the beginning of this section.
\item
The family $\Lambda \cap V_\delta$ is separable for each definable simplex $\Lambda$ of the triangulation.
\item
$\bigcup_{\delta>0} V_\delta=S$, since the unions $\bigcup_{\delta>0} (\Lambda \cap S_\delta)$ and
$\bigcup_{\delta>0} (\Lambda \cap V_\delta)$ coincide, being simultaneously either empty or equal to $\Lambda$
for each $\Lambda$.
\end{enumerate}
\end{remark}

\begin{lemma}\label{le:tele}
For $m \ge 1$ and the families $S_{\delta, \eps} $, $ V_{\delta, \eps }$, we have
$T_m(S_{\delta, \eps})=T_m(V_{\delta, \eps})$.
\end{lemma}

\begin{proof}
By the definition, for all $\delta$ and $\eps$, the sets $S_{\delta, \eps}$ and $V_{\delta, \eps}$
coincide everywhere outside the union of all simplices $\Lambda$ such that $\Lambda \cap S_\delta$ is
non-separable.
Then so do also the sets $T_m(S_{\delta, \eps})$ and $T_m(V_{\delta, \eps})$.
On the other hand, for each such $\Lambda$, we have $\Lambda \cap T_1(S_{\delta, \eps})=
\Lambda \cap T_1(V_{\delta, \eps})= \Lambda$, and therefore $\Lambda \cap T_m(S_{\delta, \eps})=
\Lambda \cap T_m(V_{\delta, \eps})= \Lambda$.
Hence, $T_m(S_{\delta, \eps})=T_m(V_{\delta, \eps})$.
\end{proof}

\begin{theorem}
\label{thm:homotopy}
If $\dim K \le 2$ then for every $1 \le k \le m$, there are epimorphisms
$$\psi_k:\> \pi_k(T_m(S_{\delta, \eps})) \to \pi_k(S),$$
$$\varphi_k:\> H_k(T_m(S_{\delta, \eps})) \to H_k(S)$$
which are isomorphisms for all $k \le m-1$.
In particular, if $m \ge \dim S$, then $T_m(S_{\delta, \eps})$ is homotopy equivalent to $S$.
\end{theorem}

\begin{proof}
Construct the family $\{V_{\delta, \eps}\}$ for $\{S_{\delta,\eps}\}$ as described above.
Since $\bigcup_{\delta>0} V_\delta=S$, and each family $\Lambda \cap V_\delta$ is separable
(see Remark~\ref{re:tele}, (ii)), there exist, by Proposition~\ref{pr:epi}, (ii), epimorphisms
$$\psi_k:\> \pi_k(T_m(V_{\delta, \eps})) \to \pi_k(S),$$
$$\varphi_k:\> H_k(T_m(V_{\delta, \eps})) \to H_k(S)$$
which are isomorphisms for all $k \le m-1$.
By Lemma~\ref{le:tele}, we have $T_m(S_{\delta, \eps})=T_m(V_{\delta, \eps})$, which implies the theorem.
\end{proof}

\bibliographystyle{abbrv}
\bibliography{master}

\def\cprime{$'$}
\begin{thebibliography}{10}

\bibitem{BGV2}
S.~Basu, A.~Gabrielov, and N.~Vorobjov.
\newblock Monotone functions and maps.
\newblock {\em Rev. R. Acad. Cienc. Exactas F\'\i s. Nat. Ser. A Math. RACSAM},
  107(1):5--33, 2013.

\bibitem{BGV_JEMS11}
S.~Basu, A.~Gabrielov, and N.~Vorobjov.
\newblock Semi-monotone sets.
\newblock {\em J. Eur. Math. Soc. (JEMS)}, 15(2):635--657, 2013.

\bibitem{Birbrair}
L.~Birbrair.
\newblock Lipschitz geometry of curves and surfaces definable in {$O$}-minimal
  structures.
\newblock {\em Illinois J. Math.}, 52(4):1325--1353, 2008.

\bibitem{Bjorner}
A.~Bj{\"o}rner.
\newblock Nerves, fibers and homotopy groups.
\newblock {\em J. Combin. Theory Ser. A}, 102(1):88--93, 2003.

\bibitem{BCR}
J.~Bochnak, M.~Coste, and M.-F. Roy.
\newblock {\em G\'eom\'etrie alg\'ebrique r\'eelle (Second edition in english:
  Real Algebraic Geometry)}, volume 12 (36) of {\em Ergebnisse der Mathematik
  und ihrer Grenzgebiete [Results in Mathematics and Related Areas ]}.
\newblock Springer-Verlag, Berlin, 1987 (1998).

\bibitem{Michel2}
M.~Coste.
\newblock {\em An introduction to o-minimal geometry}.
\newblock Istituti Editoriali e Poligrafici Internazionali, Pisa, 2000.
\newblock Dip. Mat. Univ. Pisa, Dottorato di Ricerca in Matematica.

\bibitem{GV07}
A.~Gabrielov and N.~Vorobjov.
\newblock Approximation of definable sets by compact families, and upper bounds
  on homotopy and homology.
\newblock {\em J. Lond. Math. Soc. (2)}, 80(1):35--54, 2009.

\bibitem{Hironaka_book}
H.~Hironaka.
\newblock {\em Introduction to real-analytic sets and real-analytic maps}.
\newblock Istituto Matematico ``L. Tonelli'' dell'Universit\`a di Pisa, Pisa,
  1973.
\newblock Quaderni dei Gruppi di Ricerca Matematica del Consiglio Nazionale
  delle Ricerche.

\bibitem{La10}
D.~Lazard.
\newblock C{AD} and topology of semi-algebraic sets.
\newblock {\em Math. Comput. Sci.}, 4(1):93--112, 2010.

\bibitem{Lion-Rolin}
J.-M. Lion and J.-P. Rolin.
\newblock Th\'eor\`eme de pr\'eparation pour les fonctions
  logarithmico-exponentielles.
\newblock {\em Ann. Inst. Fourier (Grenoble)}, 47(3):859--884, 1997.

\bibitem{RS}
C.~P. Rourke and B.~J. Sanderson.
\newblock {\em Introduction to piecewise-linear topology}.
\newblock Springer Study Edition. Springer-Verlag, Berlin, 1982.
\newblock Reprint.

\bibitem{Dries}
L.~van~den Dries.
\newblock {\em Tame topology and o-minimal structures}, volume 248 of {\em
  London Mathematical Society Lecture Note Series}.
\newblock Cambridge University Press, Cambridge, 1998.

\bibitem{MMV}
L.~van~den Dries, A.~Macintyre, and D.~Marker.
\newblock The elementary theory of restricted analytic fields with
  exponentiation.
\newblock {\em Ann. of Math. (2)}, 140(1):183--205, 1994.

\bibitem{VDD-Speissegger}
L.~van~den Dries and P.~Speissegger.
\newblock O-minimal preparation theorems.
\newblock In {\em Model theory and applications}, volume~11 of {\em Quad.
  Mat.}, pages 87--116. Aracne, Rome, 2002.

\end{thebibliography}

\end{document}